\documentclass[10pt]{amsart}
\usepackage{amssymb}
\usepackage[titletoc,title]{appendix}
\numberwithin{equation}{section} 
\usepackage{dsfont}
\usepackage{amscd}
\usepackage[mathscr]{euscript} % for the power-set P
\usepackage[all]{xy}
\usepackage{url}
\usepackage{stmaryrd}
\usepackage{comment}
\usepackage[retainorgcmds]{IEEEtrantools}
\usepackage{hyperref}
\usepackage{graphicx}
\usepackage{mathtools}
\usepackage{color}
\usepackage{centernot}
\usepackage{marvosym}
\usepackage{tikz}

\makeatletter
\def\author@andify{%
  \nxandlist {\unskip ,\penalty-1 \space\ignorespaces}%
    {\unskip {} \@@and~}%
    {\unskip \penalty-2 \space \@@and~}%
}

\title{Elementary Equivalence Theorem for PAC structures}

\author[J. DOBROWOLSKI]{Jan Dobrowolski$^{\spadesuit}$}
\thanks{$^{\spadesuit}$
The first author is supported  by European Union's Horizon 2020 research
and innovation  programme under the Marie Sklodowska-Curie grant agreement No 705410, and by the Foundation for Polish Science (FNP)}
\address{$^{\spadesuit}$Instytut Matematyczny, Uniwersytet Wroc\l awski, Wroc\l aw, Poland\newline
\indent {\em and}\newline
\indent School of Mathematics, University of Leeds, Leeds, UK}
\email{dobrowol@math.uni.wroc.pl}
\email{J.Dobrowolski@leeds.ac.uk}
\urladdr{http://www.math.uni.wroc.pl/~dobrowol/}

\author[D. M. HOFFMANN]{Daniel Max Hoffmann$^{\dagger}$}
\thanks{$^{\dagger}$SDG. The second author is supported by the Narodowe Centrum Nauki grants no. 2016/21/N/ST1/01465,
and 2015/19/B/ST1/01150.}
\address{$^{\dagger}$ Instytut Matematyki\\
Uniwersytet Warszawski\\
Warszawa\\
Poland
\newline \indent {\em and}\newline
\indent \hspace{1mm} Department of Mathematics \\ University of Notre Dame \\ Notre Dame \\ IN \\ USA}
\email{daniel.max.hoffmann@gmail.com}
\urladdr{https://www.researchgate.net/profile/Daniel\_Hoffmann8}

\author[J. LEE]{Junguk Lee$^{\ast}$}
\thanks{$^{\ast}$The third author is supported by the Narodowe Centrum Nauki grant no. 2016/22/E/ST1/00450}
\address{$^{\ast}$Instytut Matematyczny\\
Uniwersytet Wroc{\l}awski\\
Wroc{\l}aw\\
Poland}
\email{jlee@math.uni.wroc.pl} 
\urladdr{https://sites.google.com/site/leejunguk0323/}

\DeclareMathOperator{\acl}{acl} \DeclareMathOperator{\dcl}{dcl} 
 \DeclareMathOperator{\aut}{Aut} \DeclareMathOperator{\id}{id}

\DeclareMathOperator{\dom}{dom} \DeclareMathOperator{\gal}{Gal}

 \DeclareMathOperator{\theo}{Th}\DeclareMathOperator{\alg}{alg}

 \DeclareMathOperator{\eq}{eq}
 
\DeclareMathOperator{\tp}{tp}

\DeclareMathOperator{\mor}{Mor}

\DeclareMathOperator{\ddf}{DF}\DeclareMathOperator{\dcf}{DCF}\DeclareMathOperator{\scf}{SCF}
\DeclareMathOperator{\acf}{ACF}

\DeclareMathOperator{\qftp}{qftp}

\DeclareMathOperator{\res}{res}

\newtheorem{theorem}{Theorem}[section]
\newtheorem{prop}[theorem]{Proposition}
\newtheorem{lemma}[theorem]{Lemma}
\newtheorem{cor}[theorem]{Corollary}
\newtheorem{fact}[theorem]{Fact}
\theoremstyle{definition}
\newtheorem{definition}[theorem]{Definition}
\newtheorem{example}[theorem]{Example}
\newtheorem{remark}[theorem]{Remark}
\newtheorem{question}[theorem]{Question}

\newtheorem{notation}[theorem]{Notation}
\theoremstyle{remark}
\newtheorem*{theorem*}{Theorem}
\newtheorem*{cor*}{Corollary}

\theoremstyle{definition}
\theoremstyle{definition}

\theoremstyle{definition}

\theoremstyle{remark}

\makeatletter
\providecommand*{\cupdot}{%
  \mathbin{%
    \mathpalette\@cupdot{}%
  }%
}
\newcommand*{\@cupdot}[2]{%
  \ooalign{%
    $\m@th#1\cup$\cr
    \sbox0{$#1\cup$}%
    \dimen@=\ht0 %
    \sbox0{$\m@th#1\cdot$}%
    \advance\dimen@ by -\ht0 %
    \dimen@=.5\dimen@
    \hidewidth\raise\dimen@\box0\hidewidth
  }%
}

\providecommand*{\bigcupdot}{%
  \mathop{%
    \vphantom{\bigcup}%
    \mathpalette\@bigcupdot{}%
  }%
}
\newcommand*{\@bigcupdot}[2]{%
  \ooalign{%
    $\m@th#1\bigcup$\cr
    \sbox0{$#1\bigcup$}%
    \dimen@=\ht0 %
    \advance\dimen@ by -\dp0 %
    \sbox0{\scalebox{2}{$\m@th#1\cdot$}}%
    \advance\dimen@ by -\ht0 %
    \dimen@=.5\dimen@
    \hidewidth\raise\dimen@\box0\hidewidth
  }%
}
\makeatother

\def\Ind#1#2{#1\setbox0=\hbox{$#1x$}\kern\wd0\hbox to 0pt{\hss$#1\mid$\hss}
\lower.9\ht0\hbox to 0pt{\hss$#1\smile$\hss}\kern\wd0}

\def\ind{\mathop{\mathpalette\Ind{}}}

\def\notind#1#2{#1\setbox0=\hbox{$#1x$}\kern\wd0
\hbox to 0pt{\mathchardef\nn=12854\hss$#1\nn$\kern1.4\wd0\hss}
\hbox to 0pt{\hss$#1\mid$\hss}\lower.9\ht0 \hbox to 0pt{\hss$#1\smile$\hss}\kern\wd0}

\def\checkmark{\tikz\fill[scale=0.4](0,.35) -- (.25,0) -- (1,.7) -- (.25,.15) -- cycle;}

\def\y{\bar{y}}
\def\z{\bar{z}}
\def\a{\bar{a}}

\def\d{\bar{d}}

\def\f{\bar{f}}

\begin{document}

\newcommand{\ov}{\overline}
\newcommand{\FC}{\mathfrak{C}}

\newcommand{\twoc}[3]{ {#1} \choose {{#2}|{#3}}}
\newcommand{\thrc}[4]{ {#1} \choose {{#2}|{#3}|{#4}}}
\newcommand{\Rr}{{\mathds{R}}}
\newcommand{\Kk}{{\mathds{K}}}

\newcommand{\dlog}{\mathrm{ld}}
\newcommand{\ga}{\mathbb{G}_{\rm{a}}}
\newcommand{\gm}{\mathbb{G}_{\rm{m}}}
\newcommand{\gaf}{\widehat{\mathbb{G}}_{\rm{a}}}
\newcommand{\gmf}{\widehat{\mathbb{G}}_{\rm{m}}}
\newcommand{\gdf}{\mathfrak{g}-\ddf}
\newcommand{\gdcf}{\mathfrak{g}-\dcf}
\newcommand{\fdf}{F-\ddf}
\newcommand{\fdcf}{F-\dcf}
\newcommand{\mw}{\scf_{\text{MW},e}}

\newcommand{\BC}{{\mathbb C}}

\newcommand{\CC}{{\mathcal C}}
\newcommand{\CG}{{\mathcal G}}
\newcommand{\CL}{{\mathcal L}}
\newcommand{\CN}{{\mathcal N}}
\newcommand{\CS}{{\mathcal S}}
\newcommand{\CU}{{\mathcal U}}
\newcommand{\CF}{{\mathcal F}}
\newcommand{\CP}{{\mathcal P}}
\newcommand{\CI}{{\mathcal I}}

\begin{abstract}
We generalize a well-known theorem binding the elementary equivalence relation on the level of PAC fields and the isomorphism class of their absolute Galois groups. Our results concern two cases: saturated PAC structures and non-saturated PAC structures. 
%Finally, we provide an invariant of the theory of existentially closed substructures in a stable monster model.
\end{abstract}
\maketitle

%\tableofcontents
\section{Introduction}
%It is well known that countable $\omega$-categorical structures $M$ and $N$ are bi-inter\-pre\-table if and only if $\aut(M)$ and $\aut(N)$ are isomorphic as topological groups. Similarly, in the case when, additionally, $M$ and $N$ have  the same universe, $M$ and $N$ are bidefinable if and only if $\aut(M)$ and $\aut(N)$ are isomorphic (i.e. equal as permutation groups). 
%In this article, we establish a link between structures and their automorphism groups in the style of the aforementioned well-known facts.
\emph{Pseudo-algebraically closed} (PAC) fields were extensively studied in the second half of the 20th century. They were ``discovered" in \cite{Ax1} and \cite{Ax2}, but the name ``PAC fields" was given in \cite{FREY}. A field $K$ is PAC if and only if each nonempty absolutely irreducible $K$-variety has a $K$-rational point. Equivalently, it is existentially closed in every regular extension (compare with Definition \ref{def:PAC}). In \cite{ershov1980} and in \cite{cherlindriesmacintyre}, authors propose the name ``regularly closed fields", which in our case is more appropriate, since algebraically closed structures are not necessarily PAC structures in the sense of our Definition \ref{def:PAC} (thus algebraically closed structures can be non pseudo-algebraically closed). Moreover, because there is no useful model-theoretic generalization of the notion of separable extension of fields, we are forced to work only with definably closed substructures, which correspond to perfect fields. Therefore our definition of a PAC structure implies being definably closed (and being a perfect field, in the case of fields).

PAC fields are very attractive to model theorists (e.g. \cite{ershov1980}, \cite{CDM81}, \cite{cherlindriesmacintyre}, \cite{ChaPil}, \cite{chatzidakis2002}, 
\cite{chahru04}), since their logical and algebraic structure is, to a large extent, controlled by their absolute Galois groups. 
A key property of PAC fields - so-called ``Elementary Equivalence Theorem" - is stated in Theorem 20.3.3 in \cite{FrJa} and in Proposition 33 in \cite{cherlindriesmacintyre}. Roughly speaking, two PAC fields have the same first order theory provided they have isomorphic Galois groups. The inverse system of finite quotients of a Galois group can be organized into a first order structure and thus there exists a way of assigning a first order theory to a given absolute Galois group. In Proposition 33 in \cite{cherlindriesmacintyre}, both theories, the theory of a PAC field and the theory of its absolute Galois group, were related to each other.
The connection between the two theories is even more sophisticated, as observed in \cite{chatzidakis2017}, where the author provides a path between independence on the level of a PAC field and on the level of an absolute Galois group (see also \cite{HL1}).
Moreover, because of the newly discovered links between PAC fields and the notion of an NSOP$_1$ structure, PAC fields have been studied very recently in the model-theoretic neostability context (e.g. model theory of Frobenius fields in \cite{kaplanramsey2017}).

On the other hand, the notion of a PAC field was generalized to the context of strongly minimal theories in \cite{manuscript}, and then, to the context of stable theories in \cite{PilPol}. We use yet
a different and slightly more general
definition of a PAC structure, which was given in \cite{Hoff3} (see Section 3.1 in \cite{Hoff3}, which compares all known to us definitions of a PAC structure). In addition to the examples of PAC structures introduced in \cite{PilPol}, a general method of obtaining PAC structures 
as existentially closed structures equipped with 
%as fixed points of 
certain group actions is described in \cite{Hoff3}.

In \cite{Polkowska}, it is shown that under the assumption that \emph{PAC is a first order property} (Definition 3.3 in \cite{PilPol}, compare to Definition \ref{def:PAC.first.order}) the theory of bounded (saturated) PAC structures (in the case of a stable theory) is simple. This extends the results from \cite{manuscript}, where, in the case of a strongly minimal theory, bounded PAC structures are proven to be supersimple of SU-rank $1$ (``Bounded" means that the absolute Galois group is small as a profinite group). Hence - similarly as for PAC fields - there is an interesting connection between the
model-theoretic properties of a PAC structure and the complexity of its absolute Galois group. 
This phenomenon has motivated our interest in whether the Elementary Equivalence Theorem for PAC fields can be generalized to the class of PAC structures, and we have achieved such a generalization in two variants: 
for saturated PAC structures in Proposition \ref{lemma2023} (and Corollary \ref{cor2033})
and 
for non-saturated PAC structures in Theorem \ref{thm:elementary_invariance} (see also Corollary \ref{final_cor}).

Several technical difficulties arise in our study due to the absence of some of the tools available in the PAC fields case, such as the notion of a homomorphism of algebras, on which a key ingredient - Lemma 20.1.1a) in \cite{FrJa} - was based. Another problem one needs to deal with is that, in the case of a many-sorted language, the behavior of absolute Galois groups when taking ultraproducts of structures is much more complicated than in the case of one-sorted structures (such as fields).

The paper is organized as follows:

In Section \ref{sec:definitions}, we provide definitions and basic facts about PAC structures in the stable context. 
%We analyze there what is needed to obtain a saturated elementary extension of a PAC structure which is also a PAC structure.

In Section \ref{sec:saturated}, we prove the first of our main results, Proposition \ref{lemma2023}, where we are assuming that PAC substructures are somewhat saturated.
Previous results on PAC substructures in the stable context (\cite{PilPol}, \cite{Polkowska}) also assumed saturation, hence our result might be seen as a continuation of the preceding research in the subject.
%the assumption about saturation in Proposition \ref{lemma2023} seems to be quite natural. 
%We end Section \ref{sec:saturated} with an observation about relation between Proposition \ref{lemma2023} and Lascar Galois groups.

%In Section \ref{sec:sort_preserving}, we introduce a notion of \emph{sort-preservation} which is crucial in the proof of the most technical part of this paper, mainly of Lemma \ref{lemma2031}. 
%Proposition \ref{prop:bounded_sort_preserving} shows that \emph{sort-preservation} is more common at lower levels of complexity of considered Galois groups.
In Section \ref{sec:primitive_element} and Section \ref{sec:non_saturated_PAC}, we consider general PAC structures, not necessarily saturated. We reduce the general case to the saturated case by taking ultraproducts. In Section \ref{sec:primitive_element}, we briefly introduce the primitive element theorem for structures to handle ultraproducts of Galois groups of (many sorted) structures.
The central point of Section \ref{sec:non_saturated_PAC} is the proof of Lemma \ref{lemma2031},
which allows us to lift isomorphisms between absolute Galois groups of structures to absolute Galois groups of ultraproducts of these structures. Lemma \ref{lemma2031} is used in the proof of the second main result of this paper, Theorem \ref{thm:elementary_invariance}. Theorem \ref{thm:elementary_invariance} is the expected Elementary Equivalence Theorem for Structures.

In the appendix, we study several variants of the notion of a sorted isomorphism considered in Section \ref{sec:non_saturated_PAC}; in particular, we notice that the results of Section \ref{sec:non_saturated_PAC} remain true in a bigger generality, obtained by a weakening of the assumption of sortedness.

%The goal of this paper was to achieve Elementary Equivalence Theorem for Structures and, since its proof is quite long, we decided not to extend the paper too much and to
%leave the reader without examples. The reader interested in examples can find several of them in \cite{PilPol}. Moreover,
%results of \cite{Hoff3} also give us also a reasonable source of examples of PAC structures, since every existentially closed, equipped with a group action, substructure of some ambient stable structure is PAC. 
%Finally, a reader interested in applications of the main results might use simplified versions of these results provided in Section \ref{sec:EC} (which address a more common case of existentially closed substructures).

We fix a stable theory $T_0$ in a language $\mathcal{L}_0$, and we set $T:=(T_0^{\eq})^m$ which is a theory in the language $\mathcal{L}:=(\mathcal{L}_0^{\eq})^m$ (we add imaginary sorts, and then we do the Morleyisation). Note that $T$ is stable, has quantifier elimination and elimination of imaginaries.
Moreover, we fix a monster model $\mathfrak{C}\models T$ and assume that $T=\theo(\mathfrak{C})$
(in particular, we assume that $T$ is complete).

\section{PAC substructures}\label{sec:definitions}
%In this section, we describe a procedure which produces an elementary extension of a PAC structure which is PAC and saturated. However, to get such an extension, additional assumptions are needed. 
%More precisely, we use that \emph{PAC is a first order property} and that \emph{saturation over $P$ is a first order property} (Definition \ref{def:PAC.first.order}) in Fact \ref{fact:saturated_PAC_str}.
For a more detailed exposition of the notions of regularity and PAC structures, the reader may consult Section 3.1 in \cite{Hoff3}.

\begin{definition}\label{regular.def}
Let $E\subseteq A$ be small subsets of $\mathfrak{C}$. We say that $E\subseteq A$ is \emph{$\mathcal{L}$-regular} (or just \emph{regular}) if
$$\dcl(A)\cap\acl(E)=\dcl(E).$$
\end{definition}

\begin{definition}\label{def:PAC}
Assume that $M\preceq\mathfrak{C}$ and $P$ is a substructure of $M$. 
%\begin{enumerate}
%\item
 We say that \emph{$P$ is PAC in $M$} if for every regular extension $N$ of $P$ in $M$ (i.e. $N\subseteq M$ and $N$ is regular over $P$), the structure $P$ is existentially closed in $N$.
We say that \emph{$P$ is PAC} if $P$ is PAC in any $M'\succeq M$ (or equivalently $P$ is PAC in $\mathfrak{C}$).
%\end{enumerate}
\end{definition}

\begin{definition}[Definition 3.1 in \cite{PilPol}]\label{def:kPAC}
Assume that $M\preceq\mathfrak{C}$ and $P$ is a substructure of $M$.
We say that $P$ is $\kappa$-PAC substructure in $M$ if whenever $A\subseteq P$ is of the size strictly less than $\kappa$ and $p(x)$ is a (complete) stationary type over $A$ (in the sense of $M$), then $p(x)$ is realized by a tuple from $P$.
\end{definition}

\noindent Being a $\kappa$-PAC substructure implies being a PAC substructure. See Section 3.1 in \cite{Hoff3} for a detailed comparison of the above definitions. Moreover, the notion of a PAC substructure from \cite{PilPol} (which corresponds to a $|T|^+$-PAC substructure in the above sense) is different from our notion of a PAC substructure.

We state our definition of regularity in a language expanded by imaginaries (recall that $T=(T_0^{\eq})^m$), which is motivated by a correspondence to stationarity (e.g. Lemma 3.36 in \cite{Hoff3}) and the fact that we are working with Galois groups. Besides this, regularity considered in the language with imaginaries is more natural to us (e.g. Remark 3.2.(2) in \cite{Hoff3}). 
One could ask: 
is being existentially closed in all regular extensions preserved after adding imaginaries?
%maybe the notion of regularity changes after adding imaginaries, but being existentially closed in a regular extension means the same?
The following two examples show that, in a theory without EI, existential closedness of a set $A$ in its all regular extensions does not imply the same about $\dcl^{eq}(A)$
after adding imaginaries. We provide a detailed argument only in the second example, the statements in the first example can be proven similarly.

\begin{example}
Consider a saturated two-sorted structure $M=(X,Y,E,F,f)$ with sorts $X$ and $Y$, described as follows.
$E$ is an equivalence relation on $X$ with 3 classes: $X_1,X_2,X_3$ (we do NOT add them to the language).
$F$ is an equivalence relation on $X$ with all classes infinite such that
every $F$-class either is contained in $X_1$ or is contained in $X_2\cup X_3$ and has infinite intersections with both $X_2$ and $X_3$.
Also, $X_i/F$ is infinite for any $i\leq 3$. Finally, 
$Y=X/E$ and $f:X\to Y$ is the quotient map.
Let $A_0$ be a small set contained in $X_1$ intersecting infinitely many $F$-classes and having empty or infinite intersection with any $F$-class,
and put $A:=A_0\cup \{a/E\}$, where $a$ is any element of $A_0$. Then $A$ is existentially closed in its all regular extensions in $M$, 
but $A':=\dcl^{eq}(A)$ is not existentially closed in some regular extension (namely, in $\dcl^{eq}(A'\cup \{d/F\})$, where $d\in X_2$ is such that $d/F\notin A'$), i.e. it is not PAC after adding imaginaries.
\end{example}

\begin{example}
Consider a monster model $F$ of ACF$_p$, where $p$ is zero or a prime, and let $A\subseteq F$ be a small perfect PAC field.
% containing $\acl(\emptyset)$.
Consider a two-sorted structure $M:=(F,F)$, with no additional structure between the sorts (only the field structure on each of the sorts).
Let $B=(A, acl(A))$ be a substructure of $M$. Then $B$ is existentially closed in any regular extension, but 
$A':=\dcl^{eq}(B)$ is not PAC in $M^{eq}$.
\end{example}

\begin{proof}
It is easy to see that any regular extension of $B$ in $M$ intersected with each sort is a regular extension of the intersection of $B$ 
with that sort, and hence, by q.e., $B$ is existentially closed in any such extension. It remains to show that $A'$
is not PAC.
Let $\mathcal{M}$ be the structure obtained by expanding $M$
%by finite 
by sorts for finite sets of compatible finite tuples of elements of $M$ (together with the projections onto these sorts),
and by a sort containing only one element $c_0$. 
\
\\
\textbf{Claim 1.}
 $\mathcal{M}$ admits EI.
\
\\ Proof of the claim: 
It is enough to check that the assumptions of \cite[Theorem 3.1]{Jo} with $K:=M$ and $\mathcal{G}:=\mathcal{M}$ are satisfied. The first condition follows, as a definable set is contained in one of the sorts of $M$, and the condition holds in $ACF_p$ (as a definable set has only finitely many generic types). For the second condition, as in the previous example, any type $$q(x_1,\dots,x_n,y_1,\dots,y_m)$$ with $x_i$ of the first sort and $y_i$ of the second sort is definable over the union of sets of definitions of  $q_{|x_1,\dots,x_n}$ and  $q_{|y_1,\dots,y_m}$ (which exist in $ACF_p$). The third condition holds trivially again. Here ends the proof of the first claim.
Due to the claim, we may identify elements of $M^{eq}$ with elements of  $\mathcal{M}$.
Take any $a\in acl(A)\backslash A$ in the first sort and any $b\notin acl(A)$ in the second sort. Let $a=a_1,\dots, a_n$ be all conjugates of $a$ over $B$,
and let $b=b_1,\dots, b_n$ be a Morley sequence in  $tp(b/B)$. Consider the sort $M_E$ of $n$-element sets of pairs, where the first coordinate
in each pair is from the first sort, and the second coordinate is from the second sort ($E$ is the equivalence relation on tuples of such pairs identifying tuples being
permutations of each other). Put $d:=\{(a_i, b_i):i\leq n\}\in M_E$. 
Let $\psi(x_E)$ be the formula saying that there are $x_1,\dots,x_n,y_1,\dots,y_n$ such that $x_1,\dots,x_n$ are pairwise distinct realizations of
the formula isolating $tp(a/A)$, the elements  $y_1,\dots,y_n$ are pairwise distinct, and 
 $x_E=((x_1,y_1),\dots,(x_n,y_n))/E$ ($=\{(x_1,y_1),\dots,(x_n,y_n)\}$). Clearly, $\psi(x_E)$ is a formula over $A'$ in the language of $ \mathcal{M}$.
\
\\
\textbf{Claim 2.}
 $\psi(x_E)$ is not satisfied in $A'$.
\
\\ Proof of the claim: 
Consider any realization of $\psi(x_E)$; it must be of the form $\{(a_1,c_1),\dots,(a_n,c_n)\}$. Then  $c_i$'s are pairwise distinct, so every automorphism, which
moves $a_1$ and fixes $B$ and fixes the whole second sort pointwise, will move $\{(a_1,c_1),\dots,(a_n,c_n)\}$ (to a set in which $a_1$ will occur in pair with $c_i$ for some $i\neq 1$).
Thus $\{(a_1,c_1),\dots,(a_n,c_n)\}\notin dcl^{eq}(B)=A'$.
Here ends the proof of the second claim.

Clearly, $\psi(x_E)$ is satisfied by $d$, so, by the second claim, $A'$ is not existentially closed in  $D':=\dcl^{eq}(B,d)$. 
So, in order to conclude that $A'$ is not PAC in $M^{eq}$, it remains to show that $D'$ is a regular extension of $A'$.
\
\\
\textbf{Claim 3.}
Every automorphism of the first sort of $M$ over $A$ extends to an automorphism of $M$ fixing $Bd$ pointwise.
\
\\ Proof of the claim: 
Let $f$ be any automorphism of the first sort of $M$ over $A$. Then $f(a_1,\dots,a_n)=(a_{\sigma (1)},\dots,a_{\sigma(n)})$ for some $\sigma\in S_n$. As $(d_1,\dots, d_n)$ is a Morley sequence,
it is indiscernible as a set, so there is  an automorphism $g$
of the second sort over $\acl(A)$ which sends  $(b_1,\dots, b_n)$ to $(b_{\sigma (1)},\dots,b_{\sigma(n)})$. Then $f\cup g$ fixes $Bd$.
Here ends the proof of the third claim.

Suppose $e=\{(e_1,f_1),\dots,(e_m,f_m)\} \in D'\cap (acl(A'))$, where $e_i$'s are compatible tuples of elements of the first sort, and
$f_i$'s are compatible tuples of elements of the second sort. We will show that $e\in dcl(A')$. As $e\in acl(A')$, we get that each $f_i$ is in $\acl(A')$, so each $f_i$ is in $B$.
 
Consider any automorphism $h$ of 
 $\mathcal{M}$ over $A'$. By the last claim there is an automorphism of  $M$ which agrees with $h$ on the first sort and fixes $Bd$; denote its unique extension to $ \mathcal{M}$
by $h'$. Then, as both $h$ and $h'$ fix $A'$, $f_1,\dots,f_n\in A'$, and $h'$ agrees with $h$ on $e_1,\dots,e_m$, we get that $h(e)=h'(e)=e$ (the last equality holds,
as  $h'$ fixes $Bd$, so it must fix $e$). This shows that $e\in dcl(A')$. 
\end{proof}

The following definition is a different version of Definition 3.3 from \cite{PilPol}.

\begin{definition}\label{def:PAC.first.order}
We say that \emph{PAC is a first order property in $T$} if there exists a set $\Sigma$ of $\mathcal{L}$-sentences such that for any $M\models T$ and $P\subseteq M$
$$P\models \Sigma\qquad\iff\qquad P\text{ is PAC}.$$
\end{definition}

\begin{example}
By Proposition 11.3.2 from \cite{FrJa} and the discussion preceding it, we know that PAC is a first order property (in the sense of Definition \ref{def:PAC.first.order}) in ACF$_p$ for $p=0$ and for $p$ being a prime number. Moreover, the axioms provided in Proposition 5.6 in \cite{PilPol} show that, also in DCF$_0$, PAC is a first order property in the above sense (which is different from the condition ``PAC is a first order property" appearing in \cite{PilPol}). 
\end{example}

\section{The case of saturated PAC structures}\label{sec:saturated}
Results of this section assume that our PAC substructures are somewhat saturated as substructures of $\mathfrak{C}$. We will relax the assumptions about saturation in Section \ref{sec:non_saturated_PAC}.
%(which makes sense, since we have quantifier elimination in $\mathfrak{C}$). Note that the results of this section apply to $\kappa$-PAC structures and PAC structures in the sense of Definition 3.1 in \cite{PilPol}.

First, we prove an auxiliary fact, Lemma \ref{lemma2022}, which generalizes Lemma 20.2.2 from \cite{FrJa} (see also Lemma 2.1 in \cite{JardenKiehne}), and then we use it in the proof of Proposition \ref{lemma2023}. Proposition \ref{lemma2023} is one of the two main results of this paper (the other one is Theorem \ref{thm:elementary_invariance}) and generalizes Lemma 20.2.3 in \cite{FrJa}. Since Theorem \ref{thm:elementary_invariance} is called Elementary Equivalence Theorem for Structures, Proposition \ref{lemma2023} could be called ``Elementary Equivalence Theorem for Saturated Structures" - EETSS.

\begin{definition}
For any small substructures $A\subseteq B$ of $\mathfrak{C}$, we define
$$\CG(B/A):=\aut\big(\dcl(B)/\dcl(A\big)),\qquad \CG(A):=\aut\big(\acl(A)/\dcl(A)\big).$$
Note: $\CG(a/A)$ for a tuple $a$ is defined in a different way before Proposition \ref{thm:finGal_fingen}.
\end{definition}

\begin{remark}\label{rem:PAC.to.kPAC}
By Proposition 3.9 from \cite{Hoff3}, if $P$ is PAC and $\kappa$-saturated
(in the sense of the quantifier-free part of the $\mathcal{L}$-theory $\theo(P)$)
for $\kappa\geqslant |T|^+$, then $P$ is $\kappa$-PAC.
\end{remark}

\begin{lemma}[Embedding Lemma]\label{lemma2022}
Assume that
\begin{itemize}
\item $L\subseteq L'$, $M\subseteq M'$, $E\subseteq E'$, $F\subseteq F'$ are small Galois extensions in $\mathfrak{C}$,

\item $L\subseteq E$, $M\subseteq F$, $M'\subseteq F'$,
\item $L'\subseteq E'$ is regular,

\item $F$ is $\kappa$-PAC, where $\kappa\geqslant(|E|+|T|)^+$,
\item $\Phi_0\in\aut(\mathfrak{C})$ is such that $\Phi_0(L)=M$ and $\Phi_0(L')=M'$,
\item $\varphi:\mathcal{G}(F'/F)\to\mathcal{G}(E'/E)$ is a continuous group homomorphism such that
$$\xymatrix{
\CG(F'/F) \ar[r]^-{\varphi} \ar[d]_{\res} & \CG(E'/E) \ar[d]^{\res} \\
\CG(M'/M) \ar[r]_{\varphi_0} & \CG(L'/L)
}$$
where $\varphi_0(\sigma):=\Phi_0^{-1}\circ \sigma\circ\Phi_0$ for each $\sigma\in\CG(F'/F)$, is a commuting diagram.
\end{itemize}
Then there exists $\Phi\in\aut(\mathfrak{C})$ such that
\begin{itemize}
\item $\Phi|_{L'}=\Phi_0|_{L'}$,
\item $\Phi(E)\subseteq F$, $\Phi(E')\subseteq F'$,
\item $\varphi(\sigma)=\Phi^{-1}\circ \sigma\circ \Phi$ for any $\sigma\in\mathcal{G}(F'/F)$ .
\end{itemize}
Moreover, if $\varphi$ is onto and $E'=\acl(E)$, then $\Phi(E)\subseteq F$ is regular.
\end{lemma}

\begin{proof}
Without loss of generality, we assume that $L=M$, $L'=M'$, $\Phi_0=\id_{\mathfrak{C}}$ and $F'\ind_{L'} E'$. Our diagram looks as follows
$$\xymatrix{
\CG(F'/F) \ar[rr]^{\varphi} \ar[dr]_{\res} & & \CG(E'/E) \ar[dl]^{\res}\\
& \CG(L'/L)
}$$
We will finish the proof if we show the existence of $\Phi\in\aut(\mathfrak{C}/L')$ such that $\Phi(E)\subseteq F$, $\Phi(E')\subseteq F'$ and $\varphi(\sigma)=\Phi^{-1}\circ \sigma \circ\Phi$ for each $\sigma\in\CG(F'/F)$ (and the ``moreover part": if $\varphi$ is onto and $E'=\acl(E)$, then $\Phi(E)\subseteq F$ is regular).
\
\\
\\
\textbf{Part A}
\\ Consider the following extension of the previous diagram:
$$\xymatrix{
& \CG(F'E'/FE) \ar[dr]^{\res} & \\
\CG(F'/F) \ar[dr]_{\res} \ar[ur]^{i} \ar[rr]^{\varphi}& & \CG(E'/E)  \ar[dl]^{\res} \\
 &\CG(L'/L) & }$$
where $i$ is a continuous monomorphism of profinite groups, given as follows.
We have that $F'\ind_{L'}E'$ and $L'\subseteq E'$ is regular.
By Corollary 3.39 from \cite{Hoff3}, each pair
 $(\sigma,\varphi(\sigma))$, where $\sigma\in\CG(F'/F)$, extends to an automorphism $\tilde{\sigma}\in\aut(\mathfrak{C})$ such that $\tilde{\sigma}|_{F'}=\sigma$ and $\tilde{\sigma}|_{E'}=\varphi(\sigma)$. We define $i(\sigma)$ as the restriction of $\tilde{\sigma}$ to $\dcl(F',E')$.
We omit here checking that such an $i$ is a well-defined continuous monomorphism, which is straightforward.
Note that $i$ is a section of the continuous homomorphism
$$\res: \CG(F'E'/FE)\to \CG(F'/F).$$

We define $D\subseteq\dcl\big(F',E'\big)$ as the invariants of the group action of $i\big(\CG(F'/F) \big)$, i.e.
$$D:=\dcl\big(F',E'\big)^{i(\CG(F'/F) )}.$$
Since $i\big(\CG(F'/F) \big)$, as a continuous image of a profinite space, is a closed subgroup, the Galois correspondence (e.g. Fact 3.21 in \cite{Hoff3}) implies that
$$\CG(F'E'/D)=i\big(\CG(F'/F) \big)\xrightarrow {\quad\cong\quad\res\quad}\CG(F'/F).$$
\
\\
\textbf{Claim}
$F\subseteq D$ is regular.
\
\\ Proof of the claim: 
Suppose not, and take $m\in \big(D\cap\acl(F)\big)\setminus F$. Since $m\in\acl(F)\setminus F$, there exists $\sigma\in\CG(F)$ such that $\sigma(m)\neq m$.
Note that $\acl(F)\ind_{L'}E'$, and, by Corollary 3.39 from \cite{Hoff3}, the pair $(\sigma,\varphi(\sigma|_{F'}))$ extends to an element $\tilde\sigma\in\aut(\mathfrak{C})$ such that $\tilde\sigma|_{\acl(F)}=\sigma$ and $\tilde\sigma|_{E'}=\varphi(\sigma|_{F'})$. Note that $\tilde\sigma|_{\dcl(F'E')}=i(\sigma|_{F'})$.

Because $m\in D\subseteq\dcl(F',E')$, there are elements $f'\in F'$ and $e'\in E'$ and a formula $\psi$ such that $\{m\}=\psi(f',e',\mathfrak{C})$. After applying $\tilde\sigma$, we get
$$\{\sigma(m)\}=\{\tilde\sigma(m)\}=\psi\big(\tilde\sigma(f'),\tilde\sigma(e'),\mathfrak{C}\big)=
\psi\big(i(\sigma|_{F'})(f'),i(\sigma|_{F'})(e'),\mathfrak{C}\big).$$
On the other hand, since $m\in D$,
$$\psi\big(i(\sigma|_{F'})(f'),i(\sigma|_{F'})(e'),\mathfrak{C}\big)=\{i(\sigma|_{F'})(m)\}=\{m\},$$
hence $\sigma(m)=m$, a contradiction.
Here ends the proof of the claim.
\
\\
\\
\textbf{Part B}
\\ Note that
$$\CG(F'E'/DF')\subseteq\ker\Big(\res:\CG(F'E'/D)\to\CG(F'/F)\Big),$$
but this restriction is an isomorphism, hence
$$\CG(F'E'/DF')=\{ 1\}$$
and, by the Galois correspondence, it follows that
$$\dcl\big(F',E'\big)=\dcl\big(D,F'\big),$$
$$E'\subseteq\dcl\big(D,F'\big).$$

Enumerate the elements of $E'$ by $(m_i)_{i\in I}$, where $|I|<\kappa$.
For each $i\in I$ such that $m_i\in E\subseteq D$, we put $d_i=m_i\in D$ and $\psi_i(x,z)$ given as ``$x=z$", so 
$\psi_i(\mathfrak{C},d_i)=\{m_i\}$.

For each $i\in I$ such that $m_i\in E'\setminus E$, take $\d_i\in D$, $\f_i\in F'$, 
and a quantifier-free formula $\psi_i(x;\y_i,\z_i)$ such that
\begin{itemize}
	\item $\psi_i(\mathfrak{C};\f_i,\d_i)=\{m_i\}$ and
	\item $\d_i$ has the smallest length possible.
\end{itemize}
Note that $\d_i=\emptyset$ if $m_i\in F'$. For $\d=(\d_i)_{i\in I}$, we have
$$tp(\d/F)\vdash \tp(\d/F')$$
which follows from Corollary 3.35 in \cite{Hoff3}.
Let 
$$G\f:=\{\sigma(\f_i)|\ \sigma\in \CG(F'/F), i\in I\}\,\cup\,L'\subseteq F',$$
which is of the size smaller than $\kappa$. Thus $\qftp(\d/F_0)\vdash \qftp(\d/G\f)$ for some $F_0\subseteq F$ of size smaller that 
$\kappa$.
Let $F_1:=\acl(F_0)\cap F$ (which is also of the size smaller than $\kappa$), then $F_1\subseteq F$ is regular and, since $F\subseteq D$ is regular, we have that also $F_1\subseteq D$ is regular. Therefore $F_1\subseteq \dcl(F_1,\d)$ is regular and (by Lemma 3.35 in \cite{Hoff3}) $\tp(\d/F_1)$ is stationary (in the sense of $T$) and so realized in $F$ (since $F$ is $\kappa$-PAC).
%Note that $$\qftp(\d/F)\vdash \tp(\d/F)\vdash \tp(\d/G\f).$$
%Since $D$ is a regular extension of $F$, we have $\qftp(\d/F_0)\in\ST(F,\kappa,\kappa)$.
%By $\kappa$-regular saturation of $F$, we get that $\qftp(\d/F_0)$, and 
It follows that
$\qftp(\d/G\f)$, is realized by some $(\d'_i)_{i\in I}\subseteq F$. 
\
\\
\\
\textbf{Part C}
\\ By quantifier elimination in $T$, there exists $\Phi\in\aut(\mathfrak{C}/G\f)$
such that $\Phi(\d)=\d'$.
Note that $\Phi\in\aut(\mathfrak{C}/L')$. If $m_i\in E'$, then it follows that
$$\{\Phi(m_i)\}=\psi_i(\mathfrak{C},\f_i,\Phi(\d_i))=\psi_i(\mathfrak{C},\f_i,\d_i')
\subseteq F',$$
hence $\Phi(E')\subseteq F'$.
If $m_i\in E$ then the above line simplifies to:
$$\{\Phi(m_i)\}=\psi_i(\mathfrak{C},\Phi(\d_i))=\psi_i(\mathfrak{C},\d_i')
\subseteq F,$$
hence $\Phi(E)\subseteq F$. 

It is left to check that for any $m_i\in E'$ and any $\sigma\in\CG(F'/F)$ we have that
$$\Phi\big(\varphi(\sigma)\big)(m_i)=\sigma\big(\Phi(m_i)\big).$$
We start with
$$\psi_i(\mathfrak{C},\f_i,\d_i)=\{m_i\},$$
which gives, after applying $\Phi$,
$$\psi_i(\mathfrak{C},\f_i,\d_i')=\{\Phi(m_i)\}.$$
Now, we use $\sigma$ to get
$$\psi_i\big(\mathfrak{C},\sigma(\f_i),\d_i'\big)=\{\sigma\big(\Phi(m_i)\big)\}.$$
On the other hand, if we apply $\tilde{\sigma}$ (an extension of the pair $(\sigma,\varphi(\sigma))$) to $\psi_i(\mathfrak{C},\f_i,\d_i)=\{m_i\}$, we obtain
$$\psi_i\big(\mathfrak{C},\sigma(\f_i),\d_i\big)=\{\varphi(\sigma)(m_i)\}$$
(since $\d_i\subseteq D$). To finish the proof, we observe that the last line transforms, after applying $\Phi$, into
$$\psi_i\big(\mathfrak{C},\sigma(\f_i),\d_i'\big)=\{\Phi\big(\varphi(\sigma)(m_i)\big)\},$$
so $\{\Phi\big(\varphi(\sigma)(m_i)\big)\}=\{\sigma\big(\Phi(m_i)\big)\}$.

Now, we will prove the ``moreover part". Assume that $\varphi$ is onto and $E'=\acl(E)$. Suppose that there is $m\in\acl(E)$ such that $\Phi(m)\in F\cap\acl\big(\Phi(E)\big)$. We have that
$\sigma\big(\Phi(m)\big)=\Phi(m)$ for every $\sigma\in\CG(F'/F)$. Thus $\varphi(\sigma)(m)=(\Phi^{-1}\circ\sigma\circ\Phi)(m)=m$. Because $\varphi$ is onto, we obtain that $\tau(m)=m$ for every $\tau\in\CG(E)$, so $m\in E$.
\end{proof}

\begin{cor}\label{cor:embedding_lemma0}
If we set $F'=F$, $L'=L\subseteq F$ and $E'=E$ such that $L\subseteq E$ is regular, then we can embed $E$ into $F$ provided $F$ is $\kappa$-PAC.
\end{cor}

\begin{cor}\label{cor:embedding_lemma}
Assume that $L\subseteq F\subseteq\mathfrak{C}$, $F$ is an $(|T|+|L|)^{+}$-PAC substructure and $E$ are small definably closed substructures of $\mathfrak{C}$,
and there exists a continuous homomorphism $\varphi:\mathcal{G}(F)\to\mathcal{G}(E)$ such that
$$\xymatrix{\mathcal{G}(F) \ar[dr]_{\res} \ar[rr]^{\varphi} & & \mathcal{G}(E)\ar[dl]^{\res} \\
& \mathcal{G}(L) &
}$$
is a commuting diagram. Then there exists $\Phi\in\aut(\mathfrak{C}/\acl(L))$ such that $\varphi(\sigma)=\Phi^{-1}\circ\sigma\circ\Phi$ and $\Phi(E)\subseteq F$.
\end{cor}

%\begin{remark}\label{rem:kappaPAC} Note that in the above proof, we used saturation only for qf-types over $F$ which are stationary (i.e. $\qftp(\bar{d}/F)$). The same remains true for the proof  of Proposition \ref{lemma2023}. Compare to Definition 3.1 in \cite{PilPol} ($\kappa$-PAC structure and PAC structure). \end{remark}

\begin{prop}\label{lemma2023}
Assume that
\begin{itemize}
\item $K$, $L$, $M$, $E$, $F$ are small definably closed substructures of $\mathfrak{C}$,
\item $K\subseteq L\subseteq E$, $K\subseteq M\subseteq F$,
\item $F$ and $E$ are $\kappa$-PAC, where $\kappa\geqslant(|L|+|M|+|T|)^+$,
\item $\Phi_0\in\aut(\mathfrak{C}/K)$ is such that $\Phi_0(L)=M$,
\item $\varphi:\mathcal{G}(F)\to\mathcal{G}(E)$ is a continuous group isomorphism such that
$$\xymatrix{\mathcal{G}(E) \ar[d]_{\res} & \mathcal{G}(F)\ar@{->>}[l]_{\varphi} \ar[d]^{\res}\\
\mathcal{G}(L) & \mathcal{G}(M)\ar@{->>}[l]^{\varphi_0}}$$
where $\varphi_0(f):=\Phi_0^{-1}\circ f\circ\Phi_0$, is a commuting diagram.
\end{itemize}
Then $E\equiv_K F$.
\end{prop}

\begin{proof}
We will recursively construct:
\begin{itemize}
\item a tower of substructures of $E$, $L=:L_0\subseteq L_1\subseteq L_2\subseteq\ldots\subseteq E$, such that for each $i>0$ we have $L_i\preceq E$ and $|L_i|<\kappa$,

\item a tower of substructures of $F$, $M=:M_0\subseteq M_1\subseteq M_2\subseteq\ldots\subseteq F$, such that for each $i>0$ we have $M_i\preceq F$ and $|M_i|<\kappa$,

\item two sequences of automorphisms $\Phi_i,\Psi_i\in\aut(\mathfrak{C}/K)$, where $i\geqslant 0$, such that for each $i\geqslant 0$ we have
\begin{IEEEeqnarray*}{rClCrCl}
\Phi_{i}(L_{i})&\subseteq & M_{i+1},&\qquad &\Psi_i(M_i) &\subseteq & L_i, \\
\Phi_{i+1}|_{L_{i}} &=& \Phi_{i}|_{L_{i}},&\qquad & \Psi_{i+1}|_{M_i} &=& \Psi_i|_{M_i},\\
\Psi_{i+1}\Phi_{i}|_{L_{i}} &=&\id_{L_{i}}, &\qquad & \Phi_{i}\Psi_i|_{M_{i}} &=& \id_{M_{i}}.
\end{IEEEeqnarray*}
and the following diagrams, where $\varphi_i$ and $\psi_i$ are induced by $\Phi_i$ and $\Psi_i$ respectively, commute
$$\xymatrix{\mathcal{G}(E) \ar[d]_{\res} & \mathcal{G}(F)\ar[l]_-{\varphi} \ar[d]^{\res} && & \mathcal{G}(E)\ar[r]^-{\varphi^{-1}} \ar[d]_{\res} & \mathcal{G}(F) \ar[d]^{\res}\\
\mathcal{G}\big(\Phi_i(M_i)\big) & \mathcal{G}(M_i)\ar[l]^-{\psi_i^{-1}} && & \mathcal{G}(L_i)\ar[r]_{\varphi_i^{-1}} & \mathcal{G}\big(\Phi_i(L_i) \big)}$$
\end{itemize}
\
\\
\textbf{Step $0$}
\\ Structures $L_0=L$, $M_0=M$ and an automorphism $\Phi_0$ are given. We set $\Psi_0:=\Phi^{-1}$ and easily check that conditions required in our recursive construction are satisfied by $(L_0,M_0,\Phi_0,\Psi_0)$.
\
\\
\\
\textbf{Step $i\mapsto i+1$}
\\ Assume that we have already obtained $(L_0,M_0,\Phi_0,\Psi_0),\ldots,(L_i,M_i,\Phi_i,\Psi_i)$ which satisfy the aforementioned conditions. Our goal is to define $(L_{i+1},M_{i+1},\Phi_{i+1},\Psi_{i+1})$.

By the recursive assumption, we have $M_i=\Phi_i\Psi_i(M_i)\subseteq\Phi_i(L_i)$.
By Skolem-L\"owenheim theorem, we choose $M_{i+1}\preceq F$ which contains $M_i\subseteq\Phi_i(L_i)$ and is of size smaller than $\kappa$. ($\;\Phi_i(L_i)\subseteq M_{i+1}\;\;\checkmark \;$)

Since the following diagram commutes
$$\xymatrixcolsep{4pc}\xymatrix{\mathcal{G}(E)\ar@{->>}[r]^-{\res\circ\varphi^{-1}} \ar[d]_{\res} & \mathcal{G}(M_{i+1}) \ar[d]^{\res}\\
\mathcal{G}(L_i) \ar@{->>}[r]_-{\varphi_i^{-1}} & \mathcal{G}\big(\Phi_i(L_i)\big)}$$
Lemma \ref{lemma2022} assures the existence of $\Psi_{i+1}\in\aut(\mathfrak{C})$ such that
$$\Psi_{i+1}(M_{i+1})\subseteq E,\qquad\Psi_{i+1}|_{\acl\big(\Phi_i(L_i)\big)}=\Phi_i^{-1}|_{\acl\big(\Phi_i(L_i)\big)},$$
$$\text{and }\;(\res\circ\varphi^{-1})(f)=\Psi_{i+1}^{-1}\circ f\circ\Psi_{i+1}=:\psi_{i+1}(f),$$
where $f\in\mathcal{G}(E)$. We see that $\Psi_{i+1}\Phi_i|_{L_i}=\id_{L_i}.\;\;\checkmark$

Because $\Psi(M_i)\subseteq L_i$ and $\Phi_i\Psi_i|_{M_i}=\id_{M_i}$ (recursive assumption), it follows that $M_i=\Phi_i\Psi_i(M_i)\subseteq\Phi_i(L_i)$ and
$\Psi_i|_{M_i}=\Phi_i^{-1}|_{M_i}$, so
$$\Psi_{i+1}|_{M_i}=\Phi_i^{-1}|_{M_i}=\Psi_i|_{M_i}.\qquad\checkmark$$

Note that $L_i=\Psi_{i+1}\Phi_i(L_i)\subseteq\Psi_{i+1}(M_{i+1})$. Now, we use Skolem-L\"owenheim theorem to get $L_{i+1}\preceq E$ of size smaller than $\kappa$ such that $L_i\subseteq\Psi_{i+1}(M_{i+1})\subseteq L_{i+1}$. ($\;\Psi_{i+1}(M_{i+1})\subseteq L_{i+1}\;\;\checkmark \;$)

Before we define $\Phi_{i+1}$, we need to consider a commuting diagram, which summarizes the situation:
$$\xymatrix{ & & \mathcal{G}(F) \ar[dr]^{\res} \ar@/_2pc/[dl]_{\varphi} & \\
\mathcal{G}(L_{i+1})\ar[ddr]_{\res} & \mathcal{G}(E) \ar[d]_{\res} \ar[l]_{\res}\ar[ur]^{\varphi^{-1}} \ar[rr]^{\res\circ\varphi^{-1}=\psi_{i+1}}& & \mathcal{G}(M_{i+1}) \ar[dd]^{\res} \ar[dll]^{\varphi_{i+1}^{-1}} \\
& \mathcal{G}\big(\Phi_{i+1}(M_{i+1} \big) \ar[d]^{\res}& &  \\
& \mathcal{G}(L_i) \ar[rr]_{\varphi_i^{-1}} & & \mathcal{G}\big(\Phi_i(L_i) \big)
}$$
Therefore we obtain that
$$\xymatrix{\mathcal{G}(E) \ar[d]_{\res} & \mathcal{G}(F) \ar[l]_-{\varphi} \ar[d]^{\res} \\
\mathcal{G}\big(\Psi_{i+1}(M_{i+1})\big) & \mathcal{G}(M_{i+1}) \ar[l]^-{\psi_{i+1}^{-1}}
}$$
is commuting ($\;\checkmark\;$) and also that
$$\xymatrix{ \mathcal{G}(L_{i+1}) \ar[d]_{\res} & \mathcal{G}(F) \ar[l]_{\res\circ\varphi} \ar[d]^{\res} \\
\mathcal{G}\big(\Phi_{i+1}(M_{i+1})\big) & \mathcal{G}(M_{i+1}) \ar[l]^{\psi_{i+1}^{-1}}
}$$
is commuting, which allows us to use Lemma \ref{lemma2022}. There exists $\Phi_{i+1}\in\aut(\mathfrak{C})$ such that
$$\Phi_{i+1}(L_{i+1})\subseteq F,\qquad \Phi_{i+1}|_{\acl\big(\Psi_{i+1}(M_{i+1}) \big)}=\Psi_{i+1}^{-1}|_{\acl\big(\Psi_{i+1}(M_{i+1}) \big)},$$
$$\text{and }\;(\res\circ \varphi)(f)=\Phi_{i+1}^{-1}\circ f\circ \Phi_{i+1}=:\varphi_{i+1}(f)$$
for $f\in\mathcal{G}(F)$. Immediately, we obtain that
$$\Phi_{i+1}\Psi_{i+1}|_{M_{i+1}}=\id_{M_{i+1}}.\qquad\checkmark$$
Since $L_i\subseteq\Psi_{i+1}(M_{i+1})$ and $\Psi_{i+1}\Phi_i|_{L_i}=\id_{L_i}$, it follows that
$$\Phi_{i+1}|_{L_i}=\Psi_{i+1}^{-1}|_{L_i}=\Phi_i|_{L_i}.\qquad\checkmark$$
Also
$$\xymatrix{\mathcal{G}(E) \ar[d]_{\res} \ar[r]^-{\varphi^{-1}} & \mathcal{G}(F) \ar[d]^{\res}\\
\mathcal{G}(L_{i+1}) \ar[r]_-{\varphi_{i+1}^{-1}} &\mathcal{G}\big(\Phi_i(L_i) \big)
}$$
commutes ($\;\checkmark\;$). Therefore the recursion step is completed.

In particular, for each $i\geqslant 0$ we obtain
$$\Phi_{i+1}|_{L_i}=\Phi_i|_{L_i},\quad \Phi_i(L_i)\subseteq M_{i+1},\quad
M_i\subseteq\Phi_i(L_i).$$
Therefore $\Phi_{\infty}:L_{\infty}\to M_{\infty}$, where $L_{\infty}:=\bigcup L_i$, 
$M_{\infty}:=\bigcup M_i$ and $\Phi_{\infty}:=\bigcup\Phi_i|_{L_i}$, is an isomorphism over $K$. Hence $E\succeq L_{\infty}\cong_K M_{\infty}\preceq F$ and so $E\equiv_K F$.
\end{proof}

\begin{cor}\label{cor2033}
If $E$ and $F$ are $\kappa$-PAC substructures of $\mathfrak{C}$ ($\kappa\geqslant|T|^+$),
and for some definably closed $L\subseteq F\cap E$ of size strictly smaller than $\kappa$ there exists a continuous isomorphism $\varphi:\CG(F)\to\CG(E)$ such that
$$\xymatrix{\CG(F) \ar[dr]_{\res} \ar[rr]^{\varphi}& & \CG(E)\ar[dl]^{\res} \\
& \CG(L) &
}$$
is a commuting diagram, then $E\equiv_L F$.
\end{cor}

\begin{remark}
As noted in the introduction, in the framework of general stable theories we cannot use Lemma 20.1.1a) from \cite{FrJa} (based on the notion of a homomorphisms of algebras) which is the key ingredient in many proofs of PAC fields properties. However authors of \cite{PilPol} show that this key ingredient has its differential counterpart (Lemma 5.9 in \cite{PilPol}) and hence they can prove their Proposition 5.8, which also follows from Corollary \ref{cor2033} by the following argument. 

We assume that $F_1$ and $F_2$ (in the notation of Proposition 5.8 from \cite{PilPol}) are 
elementarily equivalent as pure fields, hence for some properly saturated elementary extensions $(F_1^*,\partial)\succeq(F_1,\partial)$ and $(F_2^*,\partial)\succeq(F_2,\partial)$
there exists a pure-field isomorphism $f_0:F_1^*\to F_2^*$ which extends to 
$f:(F_1^*)^{\alg}\to (F_2^*)^{\alg}$. Note that $\acl(\emptyset)\cap F_1^*$ (in the sense of DCF$_0$) is equal to the intersection of $F_1$ with the algebraic closure of the prime field (hence it is contained in constants of $\partial$), similarly $\acl(\emptyset)\cap F_2^*$ and therefore $f$ is a pure-field isomorphism between $\acl(\emptyset)\cap F_1^*$ and $\acl(\emptyset)\cap F_2^*$, and since both these subfields are contained in constants of $\partial$, $f$ is also a differential isomorphism between $\acl(\emptyset)\cap F_1^*$ and $\acl(\emptyset)\cap F_2^*$. 

We note here that $\CG(A)=\aut_{\dcf_0}(\acl(A)/A)=\aut_{\acf_0}(A^{\alg}/A)$ for any differential subfield $A$ (see e.g. Remark 8.4 in \cite{HL1}). Let $\varphi$ be the isomorphism of absolute Galois groups (in the sense of ACF) induced by $f$:
$$\xymatrix{\CG(F_1^*) \ar[d]_{\res} \ar[rr]^{\varphi}& & \CG(F_2^*)\ar[d]^{\res} \\
\CG(\acl(\emptyset)\cap F_1^*) \ar[rr]^{\varphi} && \CG(\acl(\emptyset)\cap F_2^*) 
}$$
Since the bottom arrow is induced by a differential homomorphism, we can use Proposition \ref{lemma2023} to get that $(F_1,\partial)\preceq (F_1^*,\partial)\equiv(F_2^*,\partial)\succeq(F_2,\partial)$.
\end{remark}

Now, we will note a fact which follows immediately from what has been proven until this point. In the following corollary we have ``replaced" the assumption about saturation by other assumptions: PAC is a first order property and our PAC substructures are bounded. By the main result of \cite{Polkowska} PAC substructures satisfying the assumptions of Corollary \ref{cor:bounded.first.order.PAC} are simple.

\begin{cor}\label{cor:bounded.first.order.PAC}
Suppose PAC is a first order property and 
%pure [or strict] saturation over $P$ is a first order property. 
\begin{itemize}
\item $K$, $L$, $M$, $E$, $F$ are small definably closed substructures of $\mathfrak{C}$,
\item $K\subseteq L\subseteq E$, $K\subseteq M\subseteq F$,
\item $F$ and $E$ are bounded PAC,
\item $\Phi_0\in\aut(\mathfrak{C}/K)$ is such that $\Phi_0(L)=M$,
\item $\varphi:\mathcal{G}(F)\to\mathcal{G}(E)$ is a continuous group isomorphism such that
$$\xymatrix{\mathcal{G}(E) \ar[d]_{\res} & \mathcal{G}(F)\ar@{->>}[l]_{\varphi} \ar[d]^{\res}\\
\mathcal{G}(L) & \mathcal{G}(M)\ar@{->>}[l]^{\varphi_0}}$$
where $\varphi_0(f):=\Phi_0^{-1}\circ f\circ\Phi_0$, is a commuting diagram.
\end{itemize}
Then $E\equiv_K F$.
\end{cor}

\begin{proof}
We wish to use Proposition \ref{lemma2023}, but to do this we need to replace our bounded  PAC structures $F$ and $E$ with suitably saturated ones, say $F^*\succeq F$ and $E^*\succeq E$ (see Remark \ref{rem:PAC.to.kPAC}). 
Then boundedness, by the proof of Proposition 2.5 in \cite{PilPol}, assures us that the restriction maps $\CG(F^*)\to\CG(F)$ and $\CG(E^*)\to \CG(E)$ are isomorphisms of profinite groups, so 
$$\xymatrix{\mathcal{G}(E^*) \ar[d]_{\res} & \mathcal{G}(F^*)\ar@{->>}[l]_{\varphi^*} \ar[d]^{\res}\\
\mathcal{G}(L) & \mathcal{G}(M)\ar@{->>}[l]^{\varphi_0}}$$
where $\varphi^*$ is induced by $\varphi$,
is commuting.
\end{proof}

In the upcoming Section \ref{sec:non_saturated_PAC}, we will deal with the following question.

\begin{question}
Is it possible to obtain the conclusion of Corollary \ref{cor:bounded.first.order.PAC} without the assumption about boundedness of $F$ and $E$?
\end{question}
We will answer it positively under the assumption that $\varphi$ is 
\emph{sorted} (and more generally, \emph{weakly sorted} - Definition \ref{def:sorted_map2} - hence, in particular, in the case of a finitely sorted language).

\begin{cor}\label{cor:acl_in_PAC}
Assume that $F$ is a $\kappa$-PAC substructure for some $\kappa\geqslant |T|^+$.
Suppose that $\acl(\emptyset)\subseteq F$.
Then, for every $A\subseteq \FC$ such that $|A|<\kappa$, we have that 
$\acl(A)$ embeds into $F$. In particular, every type over $\emptyset$ (in the sense of $\FC$) is realized in $F$ and so $F$ has a non-empty intersection with every sort.
\end{cor}

\begin{proof}
Let $A\subseteq\FC$ be any substructure such that $|A|<\kappa$.
By setting $L=\acl(\emptyset)$ and $E=\acl(A)$ in Corollary \ref{cor:embedding_lemma0}, we obtain the thesis.
\end{proof}

By Corollary \ref{cor:acl_in_PAC}, it turns out that among all saturated PAC structures in $\FC$, substructures which contain $\acl(\emptyset)$ are close to being an elementary substructure of $\FC$. However, even in the case of fields there are examples of saturated perfect PAC fields which contain $\acl(\emptyset)$ and are not algebraically closed:

\begin{example}
Fix a prime number $p$.
Let $K$ be an ultraproduct of $\mathbb{F}_{p^{n!}}$ where $n$ ranges over natural numbers
and $F_{p^{n!}}$ is a finite field of cardinality $p^{n!}$. Then $K$ is a pseudofinite field which contains the algebraic closure of its prime field.

Let $p(x)$ be an irreducible monic polynomial over $\mathbb{F}_p$. We have that $\mathbb{F}_{p^{n!}}$ contains a zero of $p(x)$ for all but finitely many $n$'s, and so does $K$. Therefore, we conclude that any non-constant polynomial over $\mathbb{F}_p$ has a zero $K$ and this implies that $K$ contains the algebraic closure of $\mathbb{F}_p$.

Note that any saturated $K'\succeq K$ is also a pseudofinite field and contains $\acl(\mathbb{F}_p)$.
Therefore $K'$ is a PAC field and, since $\CG(K')\cong \hat{\mathbb{Z}}$, $K'\not\models\acf$. Moreover, $K'$ is perfect, because $K\models \forall x\,\exists y\; (x=y^p)$.
\end{example}

Nevertheless, $\acl(\emptyset)$ plays an important role in describing PAC fields and, more generally, PAC structures. The following extends a standard result on PAC fields (see Corollary 20.4.2 in \cite{FrJa}). By \emph{an $e$-free PAC substructure} we mean a PAC substructure whose absolute Galois groups is isomorphic to the free profinite group on $e$ generators (similarly for \emph{an $e$-free $\kappa$-PAC substructure}).

\begin{cor}\label{cor:e_free_PAC}
Assume that $F$ and $E$ are $e$-free $\kappa$-PAC substructures for $\kappa\geqslant|T|^+$. If there is a definably closed substructure $K\subseteq F\cap E$ of size strictly smaller than $\kappa$ such that $F\cap\acl(K)\cong_K E\cap\acl(K)$, then
$F\equiv_K E$.
\end{cor}

\begin{proof}
We use Proposition 17.7.3 from \cite{FrJa} to obtain an isomorphism $\varphi:\CG(F)\to\CG(E)$ such that 
$$\xymatrix{
\CG(F) \ar[r]^{\varphi} \ar[d]_{\res} & \CG(E) \ar[d]^{\res} \\
\CG(F\cap\acl(K)) \ar[r]_{\varphi_0} & \CG(E\cap\acl(K))
}$$
commutes. The thesis follows by Proposition \ref{lemma2023}.
\end{proof}
\noindent
In particular, all $e$-free $|T|^+$-PAC substructures containing $\acl(\emptyset)$ have the same theory.

\section{Primitive Elements}\label{sec:primitive_element}
In this short section we provide an easy but useful observation, which is a generalization of the well-known \emph{Primitive Element Theorem} from the theory of fields. 
Recall that we are working with the theory 
$T=(T_0^{\eq})^m$ in the language $\mathcal{L}=(\mathcal{L}_0^{\eq})^m$.
It 
%$T$, which
eliminates quantifiers and imaginaries. 
Previously, we were assuming stability of $T$, but it is not necessary to work with a stable theory for the upcoming facts, hence we give up this assumption until Theorem \ref{thm:elementary_invariance} (i.e. $T$ might be \textbf{stable} or \textbf{unstable}). 
%However $T$ has elimination of imaginaries, it will be more convenient to formulate several results in the natural expansion $T^{\eq}$ in the language $\mathcal{L}^{\eq}$, where we add imaginary sorts and the natural projection maps from the home sort, $S_{=}$, to the imaginary sorts.
Recall that $\mathfrak{C}$ is an ambient monster model of $T$.

Let $S_1,\ldots, S_n$ be some sorts. Define an equivalence relation on $\bar{S}$,
where $\bar{S}:=S_1\times\ldots\times S_n$, by
$$\eta_{\bar{S}}(x_1,\ldots,x_n,y_1,\ldots,y_n)\equiv \bigwedge_{1\le i\le n} (x_i=y_i).$$
There is a sort $\bar{S}/\eta_{\bar{S}}$ and a $\emptyset$-definable function $\pi_{\bar{S}}:\bar{S}\to \bar{S}/\eta_{\bar{S}}$ such that
$$T\vdash\eta_{\bar{S}}(x_1,\ldots,x_n,y_1,\ldots,y_n)\;\leftrightarrow\;
\pi_{\bar{S}}(x_1,\ldots,x_n)=\pi_{\bar{S}}(y_1,\ldots,y_n).$$
For a tuple $\a=(a_1,\ldots,a_n)\in \bar{S}$, an imaginary $\a/\eta_{\bar{S}}:=\pi_{\bar{S}}(a_1,\ldots,a_n)$ is called {\em the imaginary corresponding to the tuple $\a$}. 

For each $i\le n$, there is a natural $\emptyset$-definable projection map $\pi_i:\bar{S}/\eta_{\bar{S}}\rightarrow \bar{S}$, $\a/\eta_{\bar{S}}\mapsto a_i$.
%For each $\alpha\in \bar{S}/\eta_{\bar{S}}$, it follows that $\dcl(\alpha)=\dcl\big(\pi_1(\alpha),\ldots,\pi_n(\alpha)\big)$.
For each $a\in \bar{S}/\eta_{\bar{S}}$,
it follows that $\dcl(a)=\dcl\big(\pi_1(a),\ldots,\pi_n(a)\big)$.

Assume that $F$ is a small definably closed substructure of $\mathfrak{C}$. 
For elements $a_1,\ldots,a_n\in\acl(F)$, we define the \emph{normal closure of $(a_1,\ldots,a_n)$ over $F$} in the following way
$$\CN_F(a_1,\ldots,a_n):=\dcl(F,\aut(\mathfrak{C}/F)\cdot a_1,\ldots,\aut(\mathfrak{C}/F)\cdot a_n).$$
By $\CG(a_1,\ldots,a_n/F)$ we denote $\aut(\CN_F(a_1,\ldots,a_n)/F)$, which is a finite group.

\begin{prop}\label{thm:finGal_fingen}
Let $F'$ be a finite Galois extension of $F$. Then $F'$ is finitely generated over $F$.
\end{prop}
\begin{proof}
Since $F'$ is a finite Galois extension of $F$, $\CG(F'/F)$ is finite, say $\CG(F'/F)=\{id=\sigma_0,\sigma_1,\dots,\sigma_n\}$. For each $1\leq i\leq n$ choose $a_i\in F'$ such that $\sigma_i(a_i)\neq a_i$. Then any automorphims fixing $a_1,\dots,a_n$ and $F$ pointwise must fix $F'$ pointwise, hence $F'=\dcl(F,a_1,\dots,a_n)$.
\end{proof}

\begin{fact}\label{fact:inverselimit_galoisgroup}
The profinite group
$\CG(F)$ is isomorphic to the inverse limit of $\CG(\a/F)$ 
with $\a$ varying over the set of finite tuples of elements of $\acl(F)$.
\end{fact}
\begin{proof}
Standard, e.g. by Corollary 1.1.6 in \cite{ribzal}.
\end{proof}

\begin{prop}[Primitive element theorem]\label{thm:PET}
\begin{enumerate}
\item 
Let $F'$ be a definably closed substructure, which is finitely generated over $F$.
There is an element $a\in \mathfrak{C}$ such that $F'=\dcl(F,a)$. 

\item
If $F'$ is a finite Galois extension of $F$, then there is an element $a\in \mathfrak{C}$ such that $F'=\dcl(F,a)$, hence
each $\sigma \in \CG(F'/F)$ is determined by $\sigma(a)$, and $|\CG(F'/F)|=k$, where $k$ is the number of conjugates of $a$ over $F$. 
\end{enumerate}
\end{prop}
\begin{proof}
(1) Suppose $F'$ is finitely generated over $F$ by $a_1,\ldots,a_n\in F'$, that is, $F'=\dcl(F,a_1,\ldots,a_n)$. 
Say each $a_i$ lives in a sort $S_i$.
Let $a\in\mathfrak{C}$ be the imaginary corresponding to the tuple $(a_1,\ldots,a_n)$, i.e. $a=(a_1,\ldots,a_n)/\eta_{S_1\times\ldots\times S_n}$. Since $\dcl(a)=\dcl(a_1,\ldots,a_n)$, we have $F'=\dcl(F,a)$.\\
(2) By Proposition \ref{thm:finGal_fingen}, $F'$ is finitely generated over $F$, so, by (1), there is $a\in \mathfrak{C}$ such that $F'=\dcl(F,a)$, and the conclusion follows.
%Assume moreover that $F'$ is normal over $F$. Let $a_{i,1},\ldots,a_{i,m_i}\in F'$ be the conjugates of $a_i$ over $F$. This time we take an imaginary $a\in \mathfrak{C}$ corresponding to the tuple $(a_{1,1},\ldots,a_{1,m_1},\ldots,a_{n,1},\ldots,a_{n,m_n})$. Then $a$ is a desired element.
\end{proof}

\section{The case of non-saturated PAC structures}\label{sec:non_saturated_PAC}
In this section  we obtain variants of Proposition 
\ref{lemma2023}
and Corollary \ref{cor2033} for non-saturated structures, eliminating in particular the boundedness assumption from Corollary \ref{cor:bounded.first.order.PAC} in the case where $\varphi$ is a sorted isomorphism (or a weakly sorted isomorphism, which is always the case for finitely sorted structures, see Remark \ref{finite_U} and Proposition \ref{weakly_prop}).

We start with an example illustrating what kind of issue we want to avoid.

\begin{example}
Consider a two-sorted structure $(\mathbb{C},\mathbb{C})$ (two sorts consisting of fields of complex numbers), where there is no interaction between the two sorts. Note that $(\mathbb{Q},\acl(\mathbb{Q}))\not\equiv(\acl(\mathbb{Q}),\mathbb{Q})$, but
$$\CG(\mathbb{Q},\acl(\mathbb{Q}))\cong\CG(\mathbb{Q})\times\CG(\acl(\mathbb{Q})) \cong\CG(acl(\mathbb{Q}))\times\CG(\mathbb{Q}) \cong\CG(\acl(\mathbb{Q}),\mathbb{Q}).$$
Moreover, we can pass to $(\mathbb{C},\mathbb{C})^{\eq}$ (or even to $\big((\mathbb{C},\mathbb{C})^{\eq}\big)^m$), but still there will be two elementary non-equivalent substructures with isomorphic absolute Galois groups.
\end{example}
While, quite surprisingly, such a situation cannot happen with sufficiently saturated PAC structures (by Corollary \ref{cor2033}), examples of similar flavour show the failure of a key ingredient of the present section - Lemma \ref{lemma2031} - in the case of arbitrary isomorphisms of Galois groups (see Example \ref{ex2031}). This motivates Definition \ref{sorted_is} below.

For a topological group $G$, we define $\CN(G)$ as the family of all open normal subgroups of $G$. If $J$ is a sort (or a finite tuple of sorts), then $\aut_{J}(L/K)$ denotes the image of the restriction map $\CG(L/K)\to\aut(S_J(L)/S_J(K))$, where $K\subseteq L$ is an extension of small substructures of $\FC$.

\begin{definition}\label{sorted_is}
 Assume that $F$ and $E$ are small substructures of $\mathfrak{C}$ and $\pi:\CG(F)\to\CG(E)$ is a continuous epimorphism.  We say that $\pi$ is an \emph{sorted} if 
for each $H\in\CN(\CG(E))$ and any sort $J$ we have 
$$|\aut_J(\acl(E)^{H}/E)|=|\CG(\acl(E)^{H}/E)| \implies $$
$$|\aut_J(\acl(F)^{\pi^{-1}[H]}/F)|=|\CG(\acl(F)^{\pi^{-1}[H]}/F)|.$$ We say $\pi$ is a \emph{sorted isomorphism} if $\pi$ is an isomorphism, and both $\pi$ and $\pi^{-1}$ are sorted.
\end{definition}

In the appendix we analyze several variants of the notion of a sorted isomorphism, and we notice that the results of this section remain true if we replace the assumption of sortedness by a weaker assumption of {\em weak sortedness}.

To prove the main result we need to show a fact about extending an isomorphism of absolute Galois groups to saturated extensions, where the notion of an ultraproduct will be useful (see Lemma 20.3.1 in \cite{FrJa}).

Recall that we are working with the theory 
$T=(T_0^{\eq})^m$ in the language $\mathcal{L}=(\mathcal{L}_0^{\eq})^m$,
which
eliminates quantifiers and imaginaries, and that $T$ is arbitrary (\textbf{stable} or \textbf{unstable}) until 
Theorem \ref{thm:elementary_invariance}.
%However $T$ has elimination of imaginaries, it will be more convenient to formulate several results in the natural expansion $T^{\eq}$ in the language $\mathcal{L}^{\eq}$, where we add imaginary sorts and the natural projection maps from the home sort, $S_{=}$, to the imaginary sorts.
Recall that $\mathfrak{C}$ is an ambient monster model of $T$.

Let $\BC\preceq\mathfrak{C}$ be sufficiently saturated, but smaller than the saturation of $\mathfrak{C}$, and let $I$ be an infinite index set and  $\CU$  a [non-principal] ultrafilter on $I$. We will say something happens for $\CU$-many $i$'s if it happens for all $i$'s from some set belonging to $\CU$. We say $(a_i)_{i\in I}$ is a compatible sequence if $\CU$-many of the $a_i$'s belong to the same sort.
We will consider ultraproducts of many-sorted structures (whose elements, by definition, are classes of {\bf compatible} sequences), in particular $\BC^*:=\prod_{\CU}\BC$, which is a model of $T$.

Let $(F_i)_{i\in I}$ be a family of small definably closed substructures of $\BC$. Let $F^*=\prod_{\CU} F_i$ be the ultraproduct of $F_i$'s with respect to $\CU$.
Then $F^*$ is a definably closed substructure of $\BC^*$.  
Note that $\acl(F_i)\subseteq \BC$, $\acl(F^*)\subseteq \BC^*$ and $F^*$ is a substructure of $\prod_{\CU}\acl(F_i)$.

If $a=(a_i)_{i\in I}$ with $a_i\in \BC$ is a compatible sequence, then by $a^*$ we will denote the element $a/\CU\in \BC^*$.

\begin{remark}\label{remark:acl_prod_F}
We have $F^*\subseteq \acl(F^*)\subseteq \prod_{\CU}\acl(F_i)\subseteq \BC^*$.
\end{remark}

\begin{proof}
Assume that $a^*\in\acl(F^*)\subseteq\BC^*$ where $a=(a_i)_{i\in I}$. 
For some element $e^*=(e_i)/\CU\in F^*$, an $\mathcal{L}$-formula $\theta(x,y)$, and a natural number $l$ we have that
$$\theta(\BC^*,e^*)=\CG(F^*)\cdot a^*,\qquad |\theta(\BC^*,e^*)|=l.$$
Therefore $\BC^*\models\exists^{=l}x\,\theta(x,e^*)$, which (by \L o\'s' theorem) is equivalent to: there exists $D\in\CU$ such that for every $i\in D$ we have
$\BC\models\exists^{=l}x\,\theta(x,e_i)$. On the other hand $\BC^*\models\theta(a^*,e^*)$ gives us $D'\in\CU$ such that $\BC\models\theta(a_i,e_i)$ for every $i\in D'$. Hence for every $i\in D\cap D'\in\CU$ we have that $a_i\in\acl(F_i)$ and so $a^*\in\prod_{\CU}\acl(F_i)$.
\end{proof}

\begin{lemma}\label{products}
Let  $\{F_i\;|\;i\in I\}$ be a family of small definably closed substructures of $\BC$ and let $n$ be a natural
number. Suppose $S$ is a sort and $a_i\in S(F_i)$ for $i\in I$. Set
 $F^*:=\prod_{\CU}F_i$, $A_i=\dcl(F_i,a_i)$, and $A^*=\dcl(F^*,a^*)$. Then:
\begin{enumerate}
\item ($A^*$ is a Galois extension of $F^*$ of degree $n$ and $\dcl(a^*)$ contains all conjugates of $a^*$ over $F^*$) $\iff$ ($A_i$ is a Galois extension of $F_i$ of degree $n$ and $\dcl(a_i)$ contains all conjugates of $a_i$ over $F_i$ for $\CU$-many $i$'s). 

\item If the equivalent conditions in (1) hold, then $A^*=\prod_{\CU} A_i$.

\item
Suppose $D\in \CU$ and $\sigma_i$ is an automorphism
of $A_i$ over $F_i$ for each $i\in D$. Then we define an automorphism $\lim_{\CU}\sigma_i$ of $A^*$ in the following way: for any $c=(c_i)_{i\in I}$ with $c_i\in A_i$, we put
$(\lim_{\CU}\sigma_i)(c^*)=(c_i')/\CU$ where $c_i'=\sigma_i(c_i)$ for $i\in D$ and $c_i'$ is chosen arbitrarily for $i\in I\backslash D$. Then $\lim_{\CU}\sigma_i$ is a well-defined automorphism of $A^*$ over $F^*$.\\
Moreover, if the equivalent conditions in (1) hold, then any automorphism of $A^*$ over $F^*$ is of this form.
\end{enumerate}
\end{lemma}
\begin{proof}
(1)
Assume the right-hand side first.
For each $i\in D$, let $a_{1,i},\dots,a_{n,i}$ be all distinct conjugates of $a_{1,i}:=a_i$ over $F_i$. Then by \L o\'s' Theorem
the elements $(a_{1,i})/\CU,\dots,(a_{n,i})/\CU$ are pairwise distinct and satisfy the same formulas over $F^*$, so they are all conjugate over $F^*$.
Also, if we let $c_i=\ulcorner\{a_{1,i},\dots,a_{n,i}\}\urcorner$ be the canonical parameter of the set $\{a_{1,i},\dots,a_{n,i}\}$, then $c_i\in F_i$, so
$$\ulcorner\{(a_{1,i})/\CU,\dots,(a_{n,i})/\CU\}\urcorner=(c_i)/\CU\in F^*.$$ Hence $\CG(F^*)\cdot\{a^*\}=\{(a_{1,i})/\CU,\dots,(a_{n,i})/\CU\}$.
Finally, notice that
$(a_{k,i})/\CU\in \dcl(a^*)\subseteq A^*$ for each $k\leq n$, as any automorphism $f$ fixing $a^*$ must fix $a_i$, and thus $a_{k,i}$, for $\CU$-many $i$'s.

Now, assume the left-hand side and let $a_1^*=a^*, a_2^*,\dots, a_n^*$ be all the conjugates of $a^*$ over $F^*$, where $a_k=(a_{k,i})_{i\in I}$ for each $k\leq n$.
Then by \L o\'s' Theorem we easily get that, for $\CU$-many $i$'s, $a_{1,i},\dots,a_{n,i}$ are all in $\dcl(a_{i})$, and 
$$\CG(F_i)\cdot a_i\subseteq \{a_{1,i},\dots,a_{n,i}\}.$$ It remains to show that for $\CU$-many $i$'s the elements $\{a_{1,i},\dots,a_{n,i}\}$ are conjugate over $F_i$. Indeed, if this is not the case, then, by the pigeonhole principle, there
is a proper nonempty subset $\{j_1,\dots,j_l\}$ of $\{1,\dots,n\}$ such that for $\CU$-many $i$'s the set $\{a_{j_1,i},\dots, a_{j_l,i}\}$ is definable over $F_i$, hence its canonical parameter $c_i:=\ulcorner\{a_{j_1,i},\dots, a_{j_l,i}\}\urcorner$ belongs to $F_i$. Then $(c_i)$ extends to a compatible sequence $c$, and $\ulcorner\{a^*_{j_1},\dots,a^*_{j_l}\}\urcorner=c^*\in F^*$, a contradiction to the assumption that $a_1^*=a^*, a_2^*,\dots, a_n^*$ are all conjugate over $F^*$.\\
%Suppose now that $|\CG(A^*)|>n$ (finite or infinite). Then there are is $\omega>m>n$ and pairwise distinct $b_1,\dots,b_m\in A^*$ such that $\{b_1,\dots,b_m\}=\CG(F^*)\cdot b_1$, with $b_j=(b_{j,i})/\CU$ for each $j\leq m$. Wlog $b_{1,i},\dots, b_{m,i}$ are pairwise distinct for every $i\in I$. We claim that for $\CU$-many $i$'s the elements $b_{1,i},\dots, b_{m,i}$ have the same type over $F_i$, which clearly will give a contradition to the assumption that $\deg(A_i/F_i)=n$. Indeed, if this were not the case, then, as $b_{i,1},\dots, b_{i,m}$ are algebraic over $F_i$, we could find by the pigeonhole principle a proper subet $\{i_1,\dots,i_k\}$ of $\{1,\dots,m\}$ such that for $\CU$-many $i$'s the set $b_{i_1,i},\dots,b_{i_k,i}$ is $F_i$-definable. Then $\pi_k(b_{i_1,i},\dots,b_{i_k,i})\in F_i$, so $\pi_k((b_{i_1,i})/\CU,\dots,(b_{i_k,i})/\CU\in F^*$, which contradicts the assumption that $\{b_1,\dots,b_m\}=\CG(F^*)\cdot b_1$.\\
(2) The inclusion $A^*\subseteq\prod_{\CU} A_i$ is a straightforward application of \L o\'s' Theorem. For the other inclusion, consider any $c^*\in \prod_{\CU} A_i$, where $c=(c_i)_{i\in I}$. There is some $D\in \CU$ such that $\dcl(F_i,a_i)=\dcl(F_i,a_i,c_i)$ is a finite Galois extension of $F_i$ of degree $n$ for each $i\in D$. By extending $c_i$'s we may assume that $\dcl(a_ic_i)$ contains all conjugates of $a_ic_i$ over $F_i$ for each $i\in D$. By (1) we get that $\dcl(F^*,a^*c^*)$ is a Galois extension of $F^*$ of degree $n$. As  $A^*$ is also a Galois extension of $F^*$ of degree $n$ contained in $\dcl(F^*,a^*c^*)$, we must have $A^*=\dcl(F^*,a^*c^*)$. In particular, $c^*\in A^*$.\\
(3) Checking that $\lim_{\CU}\sigma_i$ is well-defined and is a homomorphism is routine. Also, it is clear that $\lim_{\CU}\sigma_i^{-1}$ is an inverse to $\lim_{\CU}\sigma_i$, hence both of them are automorphisms of $A^*$ over $F^*$. 
For the 'moreover' part, notice that if $\tau$ is an automorphism of $A^*$ over $F^*$ sending $a^*$ to $c^*$, then (by the proof of (1)) for $\CU$-many $i$'s there is an automorphism $\tau_i$ of $A_i$ over $F_i$ sending $a_i$ to $c_i$, and $\tau=\lim_{\CU}\tau_i$.
\end{proof}

\begin{lemma}\label{lemma2031}
Let $\{E_i\;|\;i\in I\}$ and $\{F_i\;|\;i\in I\}$ be families of small definably closed substructures of $\BC$, and for each $i\in I$ let $\varphi_i:\CG(F_i)\to\CG(E_i)$ be a sorted isomorphism. We set $E^*:=\prod_{\CU}E_i$ and $F^*:=\prod_{\CU}F_i$.
Then, there is a sorted isomorphism of profinite groups $T:\CG(F^*)\to\CG(E^*)$.
\end{lemma}
\begin{proof}
%Put $\tau_i:=\varphi_i^{-1}$ for each $i\in I$.
Suppose $N=\dcl(F^*,a^*)$ is a Galois extension of $F^*$ of degree $n$, where 
$a=(a_i)_{i\in I}$. Wlog $\dcl(a^*)$ contains all conjugates of $a^*$ over $F^*$. 
Let $D_a\in \CU$ be the set of those $i\in I$ for which $F_i\subseteq \dcl(F_i,a_i)=:A_i$ is a Galois extension of degree $n$.
For each $i\in D_a$, define
$$M_i:=\acl(E_i)^{\varphi_i\big(\aut(\acl(F_i)/A_i) \big)},$$ and put $M:=\prod_{\CU}M_i$ (say $M_i=E_i$ for $i\notin D_a$).
Then, by Fact 3.20 in \cite{Hoff3}, for each $i\in D_a$  we have that $E_i\subseteq M_i$ is a Galois extension of degree $n$ and $\varphi_i$ induces an isomorphism $\varphi_i^{a_i}:\CG(A_i/F_i)\to \CG(M_i/E_i)$. 
\
\\
\textbf{Claim 1.}
a) There is a compatible sequence $b=(b_i)_{i\in I}$ such that $M_i=\dcl(E_i,b_i)$ for $\CU$-many 
$i$'s. Hence $M=\dcl(E^*,b^*)$ by Lemma \ref{products}.
%(call the set of such $i$'s $D_{a,b}$).
\\
b) Moreover, if $a^*$ is a tuple of elements of a sort $S$, then we can choose $b^*$ being a tuple from the sort $S$ as well.
(we will use part b) only in the final part of the proof of the theorem).
\
\\ Proof of Claim 1.
We will prove b), which in particular implies a). We may assume that $a_i\subseteq S$ for each $i\in I$. For any $i\in D_a$, let $H_i:=\aut(\acl(F_i)/A_i)$, an open subgroup of
$\CG(F_i)$ of index $n$.
As $a_i\subseteq S(A_i)$ and $A_i=\dcl(F_i,a_i)$, we clearly have  $$
|\aut_S(\acl(F_i)^{H_i}/F_i)|=|\aut_S(A_i/F_i)|=|\CG(A_i/F_i)|=|\CG(\acl(F_i)^{H_i}/F_i)|,$$ 
so, as $\varphi_i$ is a sorted isomorphism, we get that $$
|\aut_S(M_i/E_i)|=|\aut_S(\acl(E_i)^{\varphi_i(H_i)}/E_i)|=|\CG(\acl(E_i)^{\varphi_i(H_i)}/E_i)|=|\CG(M_i/E_i)|.$$ 
Thus the restriction homomorphism $\CG(M_i/E_i)\to \aut_S(M_i/E_i)$ has trivial kernel, so $M_i=\dcl(E_i,S(M_i))$. Hence we can choose $b_i$ to be an ($n-1$)-tuple of elements of $S(M_i)$ with $M_i=\dcl(E_i,b_i)$ (by picking, for each nontrivial $f\in \aut_S(M_i/E_i)$, one element of $S(M_i)$ moved by $f$). Then $b^*$ is an ($n-1$)-tuple of elements of $S(M)$.
Here ends the proof of Claim 1.

Now choose arbitrary $b^*=(b_i)/\CU\in M$ satisfying Claim 1a). We may assume $\dcl(b^*)$ contains all conjugates of $b^*$ over $E^*$. By Lemma \ref{products}(2), we have that $A^*=\prod_{\CU} A_i$ and $M=\prod_{\CU} M_i$. Thus, by Lemma \ref{products}(3), we can define $\varphi^{a}:\CG(N/F^{*})\to \CG(M/E^*)$ as follows. 
If $\sigma\in \CG(N/F^{*})$, then
$\sigma=\lim_{\CU}\sigma_i$ for some sequence of automorphisms $\sigma_i\in \CG(A_i/F_i)$ and we put $\varphi^{a}(\sigma)=\lim_{\CU}\varphi_i^{a_i}(\sigma_i)$. 
As the choice of the automorphisms $\sigma_i$ is unique up to a $\CU$-small set of indices, 
$\varphi^{a}$ is a well-defined map.
\
\\
\textbf{Claim 2.} Suppose $a'$ is such that $\dcl(F^*,a^*)=\dcl(F^*,a'^*)(=N)$. Then $\varphi^{a^*}=\varphi^{a'^*}$ (which hence depends only on $N$ and not on $a$, and will be denoted by $\varphi^N$). In particular, $M$ depends only on $N$ and not on $a$.
\
\\ Proof of Claim 2: Choose $\sigma=\lim_{\CU}\sigma_i\in \aut(N/F^{*})$.
For $\CU$-many $i$'s we have $\dcl(F_i,a_i')=\dcl(F_i,a_i)$, hence $\varphi_i^{a_i}=\varphi_i^{a'_i}$ (in particular, $\dom(\varphi_i^{a_i}(\sigma_i))=\dom(\varphi_i^{a_i'}(\sigma_i))$), thus $$\varphi^{a^*}(\sigma)=\lim_{\CU}\varphi_i^{a_i}(\sigma_i)=\lim_{\CU}\varphi_i^{a_i'}(\sigma_i)=\varphi^{a'^*}(\sigma).$$
Here ends the proof of Claim 2.

In an analogous way, using $\psi_i:=\varphi_i^{-1}$ in place of $\varphi_i$, we define $\psi_i^{b_i}$ and
$\psi^{b}: \CG(M/E^*)\to \CG(N/F^{*})$. This map, again, depends only on $M$ and not
on $b$, hence we call it $\psi^M$. Note that, for $a$ and $b$ as above, we have that $\psi_i^{b_i}=(\varphi_i^{a_i})^{-1}$ for $\CU$-many $i$'s.
\
\\
\textbf{Claim 3.} For $N$ and $M$ as above, $\psi^{M}$ is an inverse function of $\varphi^{N}$.
\
\\ Proof of Claim 3: 
Let $a$ and $b$ be as above. Choose any $\sigma=
\lim_{\CU}\sigma_i\in \CG(N/F^*)$. Then $$\psi^{M}\varphi^{N}(\sigma)=\psi^{b}\varphi^{a}(\sigma)=\psi^{b}(\lim_{\CU}\varphi_i^{a_i}(\sigma_i))=\lim_{\CU}\psi^{b_i}_i\varphi^{a_i}_i\sigma_i=\lim_{\CU}\sigma_i=\sigma,$$ so $\psi^{M}\varphi^{N}=id_{\CG(N/F^*)}$. Similarly one gets that $\varphi^{N}\psi^{M}=id_{\CG(M/E^*)}$. Here ends the proof of Claim 3.
\
\\
\textbf{Claim 4.} $\varphi^{N}$ is a group homomorphism.
\
\\ Proof of Claim 4: 
Take any $\sigma, \sigma'\in \CG(N/F^{*})$ with $\sigma=\lim_{\CU}\sigma_i$ and $\sigma'=\lim_{\CU}\sigma'_i$. Then
$$\varphi^N(\sigma\sigma')=\lim_{\CU}\varphi^{a_i}_i(\sigma_i\sigma_i')=
\lim_{\CU}\varphi^{a_i}_i(\sigma_i)\varphi^{a_i}_i(\sigma_i')=\lim_{\CU}\varphi^{a_i}_i(\sigma_i)\lim_{\CU}\varphi^{a_i}_i(\sigma_i')=\varphi^{N}(\sigma)\varphi^{N}(\sigma').$$ Here ends the proof of Claim 4.
\
\\ 
\textbf{Claim 5.} If $N'\supseteq N$ is another finite Galois extension of $F^*$,
%$\varphi^{N'}:\CG(N'/F^*)\to \CG(M'/E^*)$, 
then the following diagram commutes, 
%$$\varphi^N\pi_N=\pi_M\varphi^{N'},$$ 
where $\pi_N$ and $\pi_M$ are the restriction maps:
$$\xymatrix{\CG(N'/F^{*}) \ar[d]_{\pi_N} \ar[r]^-{\varphi^{N'}} & \CG(M'/E^{*}) \ar[d]^{\pi_M}\\
\CG(N/F^{*}) \ar[r]_-{\varphi^{N}} &\CG(M/E^{*})
}$$
\
\\ Proof of Claim 5:
We have that $N'=\dcl(F^{*},a^*a'^*)$ and $M'=\dcl(b^*b'^*)$ for some $a'=(a_i')_{i\in I}$ and $b'=(b_i')_{i\in I}$.
Take any $\sigma=\lim_{\CU}\sigma_i\in \CG(N'/F^{*})$. Then, as clearly taking ultralimit of automorphisms commutes with restriction, we have: 
\begin{IEEEeqnarray*}{rCl}
\pi_M\varphi^{N'}(\sigma) &=& \pi_M(\lim_{\CU}\varphi_i^{a_ia_i'}\sigma_i)=\lim_{\CU}\pi_M\varphi_i^{a_ia_i'}(\sigma_i) \\
&=^* & \lim_{\CU}\varphi_i^{a_i}((\sigma_i)_{|A_i})=\varphi^N(\lim_\CU (\sigma_i)_{|A_i}) =\varphi^N\pi_N(\sigma),
\end{IEEEeqnarray*}
where the equality '$=^*$' follows, as both $\varphi_i^{a_ia_i'}$ and $\varphi_i^{a_i}$ are induced by $\varphi_i$ and the restrictions  $\CG(F_i)\to\CG(\dcl(F_i,a_ia_i')/F_i)$ and 	$\CG(F_i)\to\CG(\dcl(F_i,a_i)/F_i)$, respectively, so $\varphi_i^{a_i}$ is induced by $\varphi_i^{a_ia_i'}$ and the restriction $\CG(\dcl(F_i,a_ia_i')/F_i)\to \CG(\dcl(F_i,a_i)/F_i)$.
 Here ends the proof of Claim 5.
 \
\\ The above claims show that the system $(\varphi^{N})$, with $N$ ranging over finite Galois extensions of $F^{*}$, induces an isomorphism $$T:=\varprojlim_N \varphi^N :\CG(F^*)\to\CG(E^*).$$
It remains to show that $T^{-1}$ and $T$ are sorted. 
\\
\textbf{Claim 6.} For any finite Galois extension $N$ of $F^*$, and $M$ defined as in the first paragraph (so $\varphi^{N}:\CG(N/F^*)\to \CG(M/E^*)$), we have: $$T(\CG(\acl(F^*)/N))=\CG(\acl(E^*)/M).$$
\
\\ Proof of Claim 6:
 If $\sigma\in \CG(\acl(F^*)/N)$, then $$(T(\sigma))_{|M}=\varphi^N(\sigma_{|N})=\varphi^N(id_N)=id_M,$$ which gives  $T(\CG(\acl(F^*)/N))\subseteq\CG(\acl(E^*)/M)$,
 and the other inclusion follows by symmetry.
 Here ends the proof of Claim 6.\
 \\
Consider any $H\in\CN(\CG(F^*))$ and put 
$N:=\acl(F^*)^{H}$. Then $N$ is a finite Galois extension of $F^*$. Consider any sort
$J$ such that $$|\aut_J(\acl(F^*)^{H}/F^*)|=|\CG(\acl(F^*)^H/F^*)|.$$ Then we can find a finite tuple $a^*=a/\CU$ of elements of $J(N)$ which is moved by any nontrivial element of $\aut_J(\acl(F^*)^{H}/F^*)$, hence by any nontrivial element of $\CG(\acl(F^*)^H/F^*)$ (as the restriction $\CG(\acl(F^*)^H/F^*)\to \aut_J(\acl(F^*)^{H}/F^*)$ has trivial kernel). Thus $N=\dcl(F^*,a^*)$. Define $M_i$ and $M$ as in the first paragraph of the proof. Then, by Claim 1b), there is $b^*\in \acl(E^*)$ such that $M=\dcl(E^*,b^*)$ and $b^*$ is a tuple of elements of 
$J(M)$. By Claim 6, $T(H)=\aut(\acl(E^*)/M)$, and hence
$$|\CG(\acl(E^*)^{T(H)}/E^*)|=|\CG(M/E^{*})|=|\aut_J(M/E^{*})|=|\aut_J(\acl(E^*)^{T(H)}/E^*)|,$$ where the second equality follows from the fact that $M=\dcl(E^*,b^*)$. We have proved that $T^{-1}$ is sorted, and it follows symmetrically that $T$ is sorted.
%By Lemma \ref{products}, $N=\dcl(N,a)=\prod=\prod_{\CU}\dcl(F_i,a_i)$ for some $a=(a_i)/\CU$, where for almost all $i$'s. 
%By Lemma \ref{prodcuts}, the set $D_0$ of indices $i\in I$ for which $F_i(b_i)$ is a Galois extension of $F_i$ of degree $n$ belongs to $\CU$.
% and $N=\acl(F^*)\cap\prod_{\CU}F_i(b_i)$
\end{proof}

\begin{comment}
\begin{lemma}[Lemma 20.3.1a) in \cite{FrJa}]\label{lemma2031}
Assume that $U:\mathcal{S}^{<\omega}\times\mathbb{N}\to \mathcal{S}^{<\omega}$.
Let $\{E_i\;|\;i\in I\}$ and $\{F_i\;|\;i\in I\}$ be families of small definably closed substructures of $\BC$, and for each $i\in I$ let $\varphi_i:\CG(F_i)\to\CG(E_i)$ be a $U$-sorted isomorphism. We set $E^*:=\prod_{\CU}E_i$ and $F^*:=\prod_{\CU}F_i$.
There is an isomorphism of profinite groups $\varphi:\CG(F^*)\to\CG(E^*)$.
Moreover,
\begin{enumerate}
\item if each $\varphi_i$ is an absolutely sort-preserving isomorphism, then $\varphi$ is also an absolutely sort-preserving isomorphism,
\item if each $\varphi_i$ is a sorted isomorphism, then $\varphi$ is also a sorted isomorphism,
\end{enumerate}
\end{lemma}

\end{comment}
\begin{remark}\label{remark:weakly.sorted.big.lemma}
A special case of the above lemma:
if for each $i\in I$ we have $E_i=E=\dcl(E)$, $F_i=F=\dcl(F)$ and $\varphi_i=\varphi$ is a sorted isomorphism, then there exists a sorted isomorphism $T:\CG(F^{\CU})\to\CG(E^{\CU})$.
\end{remark}

In the appendix, we will see (by Remark \ref{finite_U} and Proposition \ref{weakly_prop}) that the assumption that $\varphi_i$'s are sorted can be omitted in Lemma \ref{lemma2031} in the case of a finitely sorted language. However, in the case of an infinitely sorted language, Example \ref{ex2031} below shows that the sortedness assumption cannot be omitted even in the situation of Remark \ref{remark:weakly.sorted.big.lemma}.

\begin{remark}\label{rem_fin}
Assume $\mathfrak{C}$ is arbitrary, possibly without EI. Let $F$ be any substructure of $\FC$. Then $\CG(F)$ computed in $\mathfrak{C}$ is the same as computed in $\mathfrak{C}$ expanded by sorts for finite sets of compatible finite tuples (call the latter $\mathfrak{C}^{fin}$).
\end{remark}
\begin{proof}
Suppose $\sigma\in \aut(\mathfrak{C}^{fin}/F)$ is such that $\sigma_{|\acl^{\FC}(F)}=id_{\acl^\FC(F)}$. Then, for any $a=\ulcorner\{a_1,\dots,a_n\} \urcorner\in \acl^{\mathfrak{C}^{fin}}(F)$ we have that   $a_1,\dots, a_n\in \acl^\mathfrak{C}(F)$, hence $\sigma(a_i)=a_i$ and $\sigma(a)=a$. Thus, 
 the restriction $$\CG^{\mathfrak{C}^{fin}}(F)\to \CG^{\mathfrak{C}}(F)$$ has trivial kernel, and hence is an isomorphism.
\end{proof}

\begin{example}\label{ex2031}
Let $\mathfrak{C}$ be a structure with sorts $S_k$ and  $R_k$, $k\in \omega$, each $S_k$ equipped with an equivalence relation $E_k$ with all classes of cardinallity $2$ and with infinitely many classes, and with the projection $$\pi_k:S_k\to S_k/E_k=:R_k.$$
Choose pairwise distinct $a_1,a_2,\dots\in R_0$ and any $b_k\in
R_k$ for any $k>0$. Put $F=\{a_i:i>0\}$ and $E=\{b_i:i>0\}$.
Clearly $\acl(F)=\bigcup_{i>0}( \{a_i\}\cup \pi_0^{-1}(a_i))$ and  $\acl(E)=\bigcup_{i>0} (\{b_i\}\cup \pi_i^{-1}(b_i))$, and thus $$\CG(E)\cong (\mathbb{Z}/2\mathbb{Z})^\omega\cong \CG(F),$$ as each class $a_i$ and $b_i$ can be permutted by an automorphism independently.

Now let $I$ be an infinite set, and let $\CU$ be an ultrafilter on $I$. Then we have $$\CG(E^\CU)=\CG(\{\alpha(b_i):i>0\})\cong (\mathbb{Z}/2\mathbb{Z})^{\omega},$$
where $\alpha:\FC\to \FC^\CU$ is the diagonal map.
On the other hand, as $F$, and hence $F^\CU$, is contained in $R_0$, we easily get that $$\CG(F^\CU)=(\mathbb{Z}/2\mathbb{Z})^{|F^\CU|}.$$

If we choose $\CU$ to be regular, which we can do by Proposition 9.2(2) from \cite{ultraKeisler} (as $F$ is infinite), then, by Theorem 9.3 from \cite{ultraKeisler} we get that $|\CG(F^\CU)|=\omega^{|I|}$, so $\CG(F^\CU)\ncong \CG(E^\CU)$ if $|I|>2^\omega$. 
%Alternatively we can choose $\CU$ using what Daniel has digged up this morning in Chang-Keisler.

Finally, by Remark \ref{rem_fin} above, we get that $\CG(F^\CU)$ and $\CG(E^\CU)$ remain unchanged if we pass to $\FC^{fin}$. On the other hand, it is routine to check that $\FC^{fin}$ satisfies the assumptions of Theorem 3.1 from \cite{Jo}, hence $\FC^{fin}$ has EI. Consequently, $\CG(F^\CU)\ncong \CG(E^\CU)$ in $(\FC^{eq})^m$, while all assumptions of Remark \ref{remark:weakly.sorted.big.lemma} except sortedness are satisfied for $\dcl^{(\FC^{eq})^m}(F)$ and 
$\dcl^{(\FC^{eq})^m}(E)$. This shows that the assumption of sortedness cannot be ommited in Remark \ref{remark:weakly.sorted.big.lemma}.
\end{example}

%Consider $\mathfrak{C}=(S_0,S_1,\dots)$ where each sort $S_i$ is a model of ACF$_p$ (with $p$ fixed) and with no interaction between the sorts. For each $i\in \omega\backslash\{0\}=:I$, let $E_i=(K,L,L,L,\dots)$, where $L$ is algebraically closed but $K$ is not, and let $F_i=(L,L,\dots,L,K,L,L,\dots)$, where $K$ occurs in the sort $S_i$. Then $\CG(E_i)\cong\CG(K)\cong\CG(F_i)$ for each $i\in I$ but $\CG(\prod_\CU E_i	)\cong \CG(K) \ncong \{e\}\cong \CG(\prod_\CU F_i)$. \end{remark}

%\begin{example}\label{example_Z}
%Consider $\mathfrak{C}=\bigcup_{i\in \mathbb{Z}}S_i$ where each sort $S_i$ is a saturated model of ACF$_0$. Let $K$ be a substructure such that $S_i(K)$ is the field of real numbers for each $i$. For each $k\in \mathbb{Z}$ let $\varphi_k\in  \CG(K)\cong (\mathbb{Z}/2\mathbb{Z})^{\mathbb{Z}}$ be induced by the shift $i\mapsto i+k$ on $\mathbb{Z}$. Then the construction of $T$ from Lemma \ref{lemma2031} does not apply to $(\varphi_i)_{i\in \mathbb{Z}}$.
%\end{example}

Assume that $F$ and $E$ are small definably closed substructures of $\mathfrak{C}$ such that $E,F\subseteq\mathbb{C}\preceq\mathfrak{C}$ ($\mathbb{C}$ was already chosen, see the three paragraphs preceding Remark \ref{remark:acl_prod_F}). Now, we assume that $\mathbb{C}^{\CU}$ is embedded in $\mathfrak{C}$. Note that the ``diagonal map"
$$\alpha:\mathbb{C}\to\mathbb{C}^{\CU},\qquad\alpha(c)=(c_i)/\CU,$$
where $c_i=c$ for every $i\in I$, extends to an automorphism of $\mathfrak{C}$, which will be also denoted by $\alpha$. Since $T$ has quantifier elimination, we may abuse notation and introduce a ``scheme of maps" $\alpha^{\#}$ defined as follows
$$\alpha^{\#}(\sigma)=\big(\alpha\circ\sigma\circ\alpha^{-1}\big)|_{\alpha(A)},$$ 
where $\sigma$ is an automorphism of a small substructure $A$ of $\mathfrak{C}$. 
%- sometimes $\alpha^{\#}$ will be well-defined.

Assume that $\varphi:\CG(F)\to\CG(E)$ is a sorted isomorphism.
Consider $T:\CG(F^{\CU})\to\CG(E^{\CU})$ given by the proof of Lemma \ref{lemma2031} for $(\varphi_i)_{i\in I}$, where $\varphi_i=\varphi$ for each $i\in I$ (see Remark \ref{remark:weakly.sorted.big.lemma}).

\begin{lemma}\label{lemma:T_commutes}
The following diagram commutes
$$\xymatrix{
\CG(F^{\CU}) \ar[rrr]^-{T} \ar[d]_{\res}& & & \CG(E^{\CU}) \ar[d]^{\res} \\
\CG\big(\alpha(F)\big) \ar[r]_-{(\alpha^{\#})^{-1}}& \CG(F) \ar[r]_-{\varphi}& \CG(E) \ar[r]_-{\alpha^{\#}}& \CG\big(\alpha(E)\big)
}$$
\end{lemma}

\begin{proof}
It is enough to show that for any $\sigma\in\CG(F^{\CU})$ and any $b\in\acl(E)$ we have
$$\Big(\alpha^{-1}\circ \big(T(\sigma)|_{\acl(\alpha(E))}\big)\circ\alpha\Big)(b)
= \varphi\big(\alpha^{-1}\circ
\sigma|_{\acl(\alpha(F))}\circ\alpha \big)(b).$$
We may assume $\dcl(b)$ contains all conjugates of $b$ over $F$. Put $$M:=\dcl(F^{\CU},\alpha(b))=\prod_\CU \dcl(E,b),$$ where the second equality follows from Lemma \ref{products}(1). 
Let $a$ be such that $$\dcl(F,a)=\acl(F)^{\varphi^{-1}[\CG(\acl(E)/\dcl(E,b))]},$$ and $\dcl(a)$ contains all conjugates of $a$ over $F$.
By Claim 6 in the proof of Lemma \ref{lemma2031} and by Lemma \ref{products}(1) again, we get that $$N:=\acl(E^\CU)^{T^{-1}[\CG(\acl(E^\CU)/M)]}=(\dcl(F,a))^\CU=
\dcl(F^\CU,\alpha(a)).$$

Using the notation from the proof of Lemma \ref{lemma2031} with $\varphi_i=\varphi$ for each $i\in I$, we have that $\varphi^N$ is induced by $T$ and the restrictions $\CG(F^\CU)\to \CG(\dcl(F^\CU,\alpha(a))/F^{\CU})$ and $\CG(E^\CU)\to \CG(\dcl(E^\CU,\alpha(b))/E^{\CU})$.
Also, the other five maps from the diagram in the statement of the lemma can be arranged as the upper row in the following commuting diagram:
\begin{center}
\resizebox{12.5cm}{!}{
$$\xymatrix{
\CG(F^\CU) \ar[r]^-{\res} \ar[d]_{\res} & \CG(\alpha(F)) \ar[r]^-{(\alpha^{\#})^{-1}} \ar[d]_{\res} & \CG(F) \ar[r]^-{\varphi} \ar[d]^-{\res}  & \CG(E)  \ar[r]^-{\alpha^\#} \ar[d]_{\res} & \CG(\alpha(E)) \ar[d]^-{\res} 	& \CG(E^\CU) \ar[l]^-{res} \ar[d]^-{\res}\\
\CG(\alpha(a)/F^\CU ) \ar[r]_-{res} & \CG(\alpha(a)/\alpha(F)) \ar[r]_-{(\alpha^\#)^{-1}} & \CG(a/F) \ar[r]_-{\varphi^a}
&\CG(b/E) \ar[r]_-{\alpha^{\#}} & \CG(\alpha(b)/\alpha(E))
& \CG(\alpha(b)/E^\CU) \ar[l]_-{\res}}.$$}
\end{center}
Thus, replacing the maps from the upper row by the maps from the lower row, and $T$ by $\varphi^N$, we get that the desired equality is equivalent to: $$\Big(\alpha^{-1}\circ \big(\varphi^N(\sigma)|_{\dcl(\alpha(E),\alpha(b))}\big)\circ\alpha\Big)(b)
= \varphi^a\big(\alpha^{-1}\circ
\sigma|_{\dcl(\alpha(F),\alpha(a))}\circ\alpha \big)(b) .$$
We will compute the left-hand side first. By Lemma \ref{products}(3), $\sigma=\lim_\CU \sigma_i$ for some $\sigma_i\in \CG(a/F)$. As $\CG(a/F)$ is finite, we may assume there is some $\sigma_0\in \CG(a/F)$ such that $\sigma_i=\sigma_0$ for each $i\in I$. Put $\rho:=\varphi^a(\sigma_0)$. Then, by the definition of $\varphi^N$ in Lemma \ref{lemma2031}, we have $\varphi^N(\sigma)=\lim_\CU \rho$. Thus $\varphi^N(\sigma)$ sends $\alpha(b)$ to $\alpha(\rho(b))$, and so does $\varphi^N(\sigma)|_{\dcl(\alpha(E),\alpha(b))}$. This implies that $$ \Big(\alpha^{-1}\circ \big(\varphi^N(\sigma)|_{\dcl(\alpha(E),\alpha(b))}\big)\circ\alpha\Big)(b)=\rho(b).$$
On the other hand, $\sigma=\lim_\CU \sigma_0$ sends $\alpha(a)$ to
$\alpha(\sigma_0(a))$, and so does its restriction to $\dcl(\alpha(F),\alpha(a))$. Hence, the function $\alpha^{-1}\circ
\sigma|_{\dcl(\alpha(F),\alpha(a))}\circ\alpha$ sends $a$ to $\sigma_0(a)$, and so is equal to $\sigma_0$. Thus $$\varphi^a\big(\alpha^{-1}\circ
\sigma|_{\dcl(\alpha(F),\alpha(a))}\circ\alpha \big)(b)=(\varphi^a(\sigma_0))(b)=\rho(b).$$
\end{proof}

\begin{remark}
Assume that, for each $i\in I$, $K_i$ and $F_i$ are small substructures of $\mathfrak{C}$ such that $\dcl(K_i)=K_i$ and $K_i\subseteq F_i$ is regular. Then $\prod K_i/\CU\subseteq\prod F_i/\CU$ is regular.
\end{remark}

Although
regularity is preserved under taking ultraproducts,
it is not clear that being a PAC substructure is. Therefore, in the next result, we assume that ``PAC is a first order property". 
The following theorem is the second main theorem of this paper and generalizes (the perfect case of) Theorem 20.3.3 from \cite{FrJa}. Now, we assume \textbf{stability} of the theory $T$ again.

\begin{theorem}[Elementary Equivalence Theorem for Structures - EETS]\label{thm:elementary_invariance}
Suppose PAC is a first order property.
%and pure [or strict] saturation over $P$ is a first order property. 
Assume that
\begin{itemize}
\item $K$, $L$, $M$, $E$, $F$ are small definably closed substructures of $\mathfrak{C}$,
\item $K\subseteq L\subseteq E$, $K\subseteq M\subseteq F$,
%\item $F$ and $E$ are $\kappa^{+}$-saturated (\textbf{for qf-types?}), where $\kappa\geqslant\max(|L|,|M|,|T|)$,
\item $F$ and $E$ are PAC,
\item $\Phi_0\in\aut(\mathfrak{C}/K)$ is such that $\Phi_0(L)=M$,
\item $\varphi:\mathcal{G}(F)\to\mathcal{G}(E)$ is a \textbf{sorted isomorphism} such that
$$\xymatrix{
\mathcal{G}(F)\ar[r]^{\varphi} \ar[d]_{\res} & \mathcal{G}(E) \ar[d]^{\res}\\
\mathcal{G}(M)\ar[r]_{\varphi_0} & \mathcal{G}(L)
}$$
where $\varphi_0(\sigma):=\Phi_0^{-1}\circ \sigma\circ\Phi_0$, commutes.
\end{itemize}
Then $E\equiv_K F$.
\end{theorem}

\begin{proof}
%\textbf{PAC and saturation of $F^{\CU}$ and $E^{\CU}$...TO DO!}
By standard facts (e.g. Theorem 6.1.4 and Theorem 6.1.8 from \cite{ChangKeisler}), we can choose $I$ and $\CU$ such that $N^{\CU}$ will be $\kappa$-saturated for any $\mathcal{L}$-structure $N$, where
$\kappa>(|M|+|L|+|T|)^+$, so we do it.
Since PAC is a first order property, it follows that $F^{\CU}$ and $E^{\CU}$ are PAC, hence $\kappa$-PAC.

By Lemma \ref{lemma2031} and Lemma \ref{lemma:T_commutes} we obtain the following commuting diagram
$$\xymatrix{
\CG(F^{\CU}) \ar[rrr]^-{T} \ar[d]_{\res}& & & \CG(E^{\CU}) \ar[d]^{\res} \\
\CG\big(\alpha(F)\big) \ar[d]_{\res} \ar[r]^-{(\alpha^{\#})^{-1}}& \CG(F) \ar[r]^-{\varphi}\ar[d]_{\res}& \CG(E) \ar[d]^{\res}\ar[r]^-{\alpha^{\#}}& \CG\big(\alpha(E)\big) \ar[d]^{\res}\\
\CG\big(\alpha(M)\big) \ar@/_2.5pc/[rrr]_{\sigma\mapsto(\alpha\Phi_0\alpha^{-1})^{-1}\circ\sigma\circ(\alpha\Phi_0\alpha^{-1})} \ar[r]_-{(\alpha^{\#})^{-1}}& \CG(M)\ar[r]_{\varphi_0} & \CG(L) \ar[r]_-{\alpha^{\#}}& \CG\big(\alpha(L)\big)
}$$
Note that $\alpha\Phi_0\alpha^{-1}\in\aut(\mathfrak{C}/\alpha(K))$. By Proposition \ref{lemma2023}, we obtain that $F^{\CU}\equiv_{\alpha(K)}E^{\CU}$. Since $\alpha(F)\preceq F^{\CU}$ and $\alpha(E)\preceq E^{\CU}$, it follows that 	$\alpha(F)\equiv_{\alpha(K)}\alpha(E)$ and so $F\equiv_K E$.
\end{proof}

\begin{remark}
Let us note here that the above theorem remains true if we understand ``PAC" and ``PAC is a first order property" as they are defined in \cite{PilPol}.
We only need to use Definition 3.1, Definition 3.3 from \cite{PilPol}, and note the fact that
$$(\mathbb{C},F)\cong\big(\alpha(\mathbb{C}),\alpha(F)\big)\preceq(\mathbb{C},F)^{\CU}=(\mathbb{C}^{\CU},F^{\CU})$$
and the analogous fact for $E^{\CU}$.
\end{remark}

\begin{cor}\label{final_cor}
Suppose PAC is a first order property.
If the restriction map $\res:\CG(F)\to\CG(E)$, where $E\subseteq F$ are PAC structures, is a  sorted isomorphism, then $E\preceq F$. (Compare with Remark \ref{rem:bounded_sort_preserving})
\end{cor}

\begin{appendix}

\section{Sort-preserving isomorphisms of Galois groups of multi-sorted strucutres}
In this appendix, we introduce several variations of the notion of a sorted isomorphism of Galois groups of multi-sorted structures, and we compare them with each other. Then, among other observations, we note that the main results of Section \ref{sec:non_saturated_PAC} can be strengthened by replacing the sortedness assumption by weak sortedness.

%We end this section with Proposition \ref{prop:bounded_sort_preserving}, which shows that being an \emph{absolutely sort-preserving isomorphism} is not uncommon, at least if the absolute Galois groups are not too big.

It seems to us that being an \emph{absolutely sort-preserving isomorphism} (Definition \ref{def:sorted_map1} below) is a natural notion in the realm of many sorted structures. Recall that an isomorphism of profinite groups is an inverse limit of isomorphisms of finite quotients of these groups, thus it is natural to require some model-theoretic behavior on the level of each finite quotient. 

In \cite{cherlindriesmacintyre}, authors introduce a notion of \emph{complete inverse systems} for Galois groups attached to fields. Being an absolutely sort-preserving isomorphism is related to being an isomorphism of ``sorted" complete inverse systems, so to an isomorphism of complete inverse systems where sorts are named. We will study this concept in our future research (\cite{HL1}).

\begin{notation}
Let $\mathcal{S}$ be a set of sorts.
\begin{enumerate}
	\item Let $\mathcal{S}^{<\omega}$ be the set of finite tuples of elements in $\mathcal{S}$.
	\item For $J,J'\in \mathcal{S}^{<\omega}$, we write $J\leqslant J'$ if $J$ is a subtuple of $J'$
	(i.e. if $J'=(S_{1},\ldots,S_{n})$, then any $J=(S_{{s_1}},\ldots,S_{{s_{n'}}})$, where $n'\leqslant n$ and $1\leqslant s_1<\ldots<s_{n'}\leqslant n$, is a subtuple of $J'$).
	\item For $J,J'\in \mathcal{S}^{<\omega}$, we write $J^\smallfrown J'$ for the concatenation of $J$ and $J'$.
	\item For $J=(S_1,\ldots,S_n)\in \mathcal{S}^{<\omega}$, set $|J|=n$.
	\item For $J=(S_1,\ldots,S_n)\in \mathcal{S}^{<\omega}$ and a permutation $\sigma\in \mathcal{S}_n$, $\sigma(J)=(S_{\sigma(1)},\ldots,S_{\sigma(n)})$.
	\item For $J=(S_1,\ldots,S_n)\in I^{<\omega}$, we write $S_J=S_{1}\times\cdots\times S_{n}$.

%Moved to Section 'the case of...': \item In this notation, $\aut_{J}(L/K)$ is the image of the restriction map $\CG(L/K)\to\aut(S_J(L)/S_J(K))$, where $K\subseteq L$ is an extension of small substructures of $\FC$.
\end{enumerate}
\end{notation}

\begin{definition}\label{def:sorted_map1}
Assume that $F$ and $E$ are small substructures of $\mathfrak{C}$ and $\pi:\CG(F)\to\CG(E)$ is a continuous epimorphism. 
We say that $\pi$ is \emph{absolutely sort-preserving} if for each
$N\in\CN(\CG(E))$, each $J\in \mathcal{S}^{<\omega}$ and every $\sigma\in\CG(\acl(F)^{\pi^{-1}[N]}/F)$ we have that
$$\sigma|_{S_J}=\id_{S_J}\qquad\Rightarrow\qquad \pi_N(\sigma)|_{S_J}=\id_{S_J},$$
where $\pi_N:\CG(\acl(F)^{\pi^{-1}[N]}/F)\to\CG(\acl(E)^{N}/E)$ is the induced homomorphism of profinite groups.

We say that $\pi$ is an absolutely sort-preserving isomorphism if $\pi$ is an isomorphism of profinite groups such that $\pi$ and $\pi^{-1}$ are absolutely sort-preserving.
\end{definition}

\begin{remark}
Being absolutely sort-preserving may be checked over single sorts instead of over finite tuples of sorts: $\pi$ is absolutely sort-preserving if and only if
for each
$N\in\CN(\CG(E))$, each $S\in \mathcal{S}$ and every $\sigma\in\CG(\acl(F)^{\pi^{-1}[N]}/F)$ we have that
$$\sigma|_{S}=\id_{S}\qquad\Rightarrow\qquad\tilde{\pi}(\sigma)|_{S}=\id_{S}.$$
\end{remark}

\begin{remark}
\begin{enumerate}
\item Let $\sigma\in\aut(\mathfrak{C})$ and let $E$, $F$ be small substructures of $\mathfrak{C}$ with $\sigma[E]=F$. Then the automorphism $\sigma$ induces an isomorphism $\Sigma:\CG(F)\rightarrow \CG(E)$ of profinite groups, given by $\tau\mapsto \sigma^{-1}\circ \tau \circ \sigma$, which is an absolutely sort-preserving isomorphism.

\item If $\xi:\CG(E)\rightarrow \CG(F)$ and $\zeta:\CG(F)\rightarrow \CG(K)$ are absolutely sort-preserving isomorphisms, then $\zeta\circ\xi$ is an absolutely sort-preserving isomorphism and
also $\xi^{-1}$ is an absolutely sort-preserving isomorphism.
\end{enumerate}
\end{remark}

\begin{cor}\label{lemma1.7.sort_preserving}
By the conclusion of Lemma \ref{lemma2022}, the epimorphism $\varphi$ in Lemma \ref{lemma2022} is  absolutely sort-preserving.
\end{cor}

We leave the following remark without proof, but let us note that the second point follows from the first point.

\begin{remark}\label{rem:bounded_sort_preserving}
Let $T$ be stable.
\begin{enumerate}
\item
If $E\preceq F$ are definably closed and bounded, then the restriction map $\varphi:\CG(F)\to\CG(E)$ is an absolutely sort-preserving isomorphism.

\item
Assume that $L,F,E$ are definably closed substructures of $\mathfrak{C}$, $M_L,M_F,M_E\preceq\mathfrak{C}$ and $(M_F,F),(M_E,E)\succeq(M_L,L)$. Moreover, let $\varphi:\CG(F)\to\CG(E)$ be an isomorphism of profinite groups such that
$$\xymatrix{
\CG(F) \ar[rr]^{\varphi} \ar[rd]_{\res} & & \CG(E) \ar[dl]^{\res}\\
& \CG(L) &
}$$
commutes.
If $L$ is bounded, then $\varphi$ is an absolutely sort-preserving isomorphism.
\end{enumerate}
\end{remark}

Although the notion of an \emph{absolutely sort-preserving isomorphism} of absolute Galois groups
arises quite naturally, it is not easy to find situations when it occurs (e.g. Remark \ref{rem:bounded_sort_preserving}). 
Therefore in the following lines we refine (and weaken) the notion of an absolutely sort-preserving isomorpshism, so that it can be encoded by a first order theory of an absolute Galois group (as was done in \cite{HL1}).

\begin{definition}
Element $a\in\FC$ is a \emph{primitive element} of a Galois extension $K\subseteq L$ (Definition 3.14 in \cite{Hoff3}) if $L=\dcl(K,a)$. By $e(L/K)$ we denote the subset of $\FC$ of all primitive elements of the Galois extension $K\subseteq L$.
\end{definition}

Hence Proposition \ref{thm:PET} states that if for a Galois extension $K\subseteq L$ we have that $|\CG(L/K)|<\omega$, then $e(L/K)\neq\emptyset$.
For a small substructure $K$ of $\FC$ and $N\in\CN(\CG(K))$, we put 
$$\CP_e(N):=\Big\{J\in \mathcal{S}^{<\omega}\;|\;\big(\exists a\in S_J(\FC)\big)\Big(a\in e\big(\acl(K)^N/K\big)\Big)\Big\},$$
$$\CF(N):=\{J\in \mathcal{S}^{<\omega}\;|\;	|\aut_J(\acl(K)^{N}/K)|=|\CG(\acl(K)^N/K)| \}.$$

The following definition and several facts after it were formulated in an early version of \cite{HL1}.

\begin{definition}\label{def:sorted_map2}
Assume that $F$ and $E$ are small substructures of $\mathfrak{C}$ and $\pi:\CG(F)\to\CG(E)$ is a continuous epimorphism. 
\begin{enumerate}
\item We say that $\pi$ is \emph{sorted} if 
for each $N\in\CN(\CG(E))$ we have $\CF(N)\subseteq\CF(\pi^{-1}[N])$.

\item  Let $U:\mathcal{S}^{<\omega}\times\mathbb{N}\to \mathcal{S}^{<\omega}$ be a function. We say that $\pi$ is \emph{$U$-sorted} if for each $N\in\CN(\CG(E))$ and each $S\in\CP_e(N)\cap\mathcal{S}$ we have that $U(S,[\CG(F):\pi^{-1}[N]])\in\CP_e(\pi^{-1}[N])$.

\item Epimorphism $\pi$ is \emph{weakly sorted} if there exists $U:\mathcal{S}^{<\omega}\times\mathbb{N}\to \mathcal{S}^{<\omega}$ such that $\pi$ is $U$-sorted.

\end{enumerate}
We say that $\pi$ is a weakly sorted isomorphism [$U$-sorted isomorphism / sorted isomorphism] if $\pi$ is an isomorphism of profinite groups such that $\pi$ and $\pi^{-1}$ are weakly sorted [$U$-sorted / sorted]. 
\end{definition}

Clearly, the above definition is consistent with Definition \ref{sorted_is}.

\begin{remark}\label{finite_U}
If  $\mathfrak{C}$ has only finitely many sorts $S_1,\dots,S_k$, then any continuous epimorphism $\pi:\CG(F)\to\CG(E)$ is $U$-sorted for $U$ defined by: $U(S,n):=J_0^n$ for each $S \in \CS$ and $n\ge 1$, where
$J_0=S_1 \times S_2 \times \cdots \times S_k$.
\end{remark}

\begin{fact}\label{fact:implications_sorted}
Assume that $F$ and $E$ are small substructures of $\mathfrak{C}$ and $\pi:\CG(F)\to\CG(E)$ is a continuous epimorphism. We have (i) $\Rightarrow$ (ii) $\Rightarrow$ (iii) $\Rightarrow$ (iv) $\Rightarrow$ (v), where:
\begin{itemize}
\item[(i)]
$\pi$ is an absolutely sort-preserving isomorphism,
\item[(ii)]
$\pi$ is a sorted isomorphism,
\item[(iii)]
$\pi$ is sorted,
\item[(iv)]
$\pi$ is $U$-sorted for $U(J,n)=\underbrace{J^\smallfrown J^\smallfrown\ldots ^\smallfrown J}_{n\text{-many times}}$,
\item[(v)]
$\pi$ is weakly sorted.
\end{itemize}
\end{fact}

\begin{proof}
Implication (i) $\Rightarrow$ (ii) follows, since for each $N\in\CN(\CG(E))$ and each $J\in \mathcal{S}^{<\omega}$ there exists $\pi_{N,J}$ (induced by the isomorphism $\pi_N$, which is induced by $\pi$) such that
$$\xymatrix{
\CG(\acl(F)^{\pi^{-1}[N]}/F) \ar[r]^-{\pi_N} \ar[d]_{\res} & \CG(\acl(E)^N/E) \ar[d]^{\res} \\
\aut_J(\acl(F)^{\pi^{-1}[N]}/F) \ar[r]_-{\pi_{N,J}}& \aut_J(\acl(E)^N/E)
}$$
For the proof of implication (iii) $\Rightarrow$ (iv) we argue as follows.
Assume that $N\in\CN(\CG(E))$.
Actually we will show more than we need, namely that: if $J\in\CP_e(N)$ then
$$\underbrace{J^{\smallfrown}J^\smallfrown\ldots^\smallfrown J}_{l\text{-many times}}\in\CP_e(\varphi^{-1}[N]),$$
where $l:=[\CG(F):\varphi^{-1}[N]]$. 
Let us pick some $a\in e(L/E)$ for $L:=\acl(E)^N$ such that $a\in S_J$. Let $b:=\ulcorner\alpha\urcorner\in S$ for a proper sort $S$ (corresponding to the code of the tuple $a$); it follows that $b$ is also a primitive element of the extension $E\subseteq L$ so
$$\CG(L/E)\ni\sigma\mapsto\sigma(b)\in \CG(L/E)\cdot b$$
is a bijection. Therefore
$$|\aut_S(L/E)|\leqslant|\CG(L/E)|=|\CG(L/E)\cdot b|
=|\aut_S(L/E)\cdot b|\leqslant|\aut_S(L/E)|.$$
Thus both vertical maps in the following diagram are isomorphisms
$$\xymatrix{
\CG(L'/F) \ar[r]^{\pi_N} \ar[d]_{\res} & \CG(L/E)\ar[d]^{\res} \\
\aut_S(L'/F) & \aut_S(L/E)
}$$
where $L':=\acl(F)^{\varphi^{-1}[N]}$.
We recursively choose elements 
$c_1, c_2,\ldots \in L'$ such that
$$c_{k+1}\in S(L')\setminus S(\dcl(F,c_1,\ldots,c_k)).$$
Let $A_k$ denote $\dcl(F,c_1,\ldots,c_k)$.
We get a strictly increasing tower of subsets
$$S(F)\subsetneq S(A_1)\subsetneq S(A_2)\subsetneq\ldots $$
which translates into the following chain of subgroups
$$\aut_S(L'/F)\gneq\aut_S(L'/A_1)\gneq\aut_S(L'/A_2)\gneq\ldots$$
Since $|\CG(L'/F)|=l$, the group $\aut_S(L'/F)$ is finite. Therefore there exists $k'\leqslant l$ such that
$$\lbrace \id_{S(L')}\rbrace=\aut_S(L'/A_{k'})=\res^{L'}_S\big(\CG(L'/A_{k'})\big).$$
Because $\res^{L'}_S$ is an isomorphism, we obtain that $\CG(L'/A_{k'})=\lbrace\id_{L'}\rbrace$ and so
$L'=\dcl(F,c_1,\ldots,c_{k'})$ (by the Galois correspondence).
Finally, we set $c_{1}=c_1,\ldots,c_{k'}=c_{k'}, c_{k'+1}=c_{k'},\ldots,c_{l}=c_{k'}$ so
$L'=\dcl(F,c_{1},\ldots,c_{l})$. 
Since each $c_i\in S$, there exists a tuple $d_i\in S_J$ such that $c_i=\ulcorner d_i\urcorner$ and so we have
 that $L'=\dcl(F,\bar{d})$ for $\bar{d}=(d_1,\ldots,d_l)$ 
 which means that $\underbrace{J^{\smallfrown}J^\smallfrown\ldots^\smallfrown J}_{l\text{-many times}}\in\CF(\varphi^{-1}[N])$.
Proofs of the omitted implications are straightforward.
\end{proof}

The following proposition validates our notion of a sorted map (i.e. it is natural for regular extensions in the stable case).

\begin{prop}\label{fact:regular_to_sorted}
Assume that $T$ is stable, $E\subseteq F$ is regular extension of small substructures of $\FC$,
and $\pi:\CG(F)\to\CG(E)$ is the restriction map.
For each $N\in\CN(\CG(E))$, if $a\in e(\acl(E)^N/E)$, then $a\in e(\acl(F)^{\pi^{-1}[N]}/F)$. Therefore $\pi$ is sorted.
\end{prop}

\begin{proof}
Be Lemma 2.15 from \cite{Hoff4}, the map $\pi$ is onto.
Let $a\in e(\acl(E)^N/E)$ and
 $L:=\acl(E)^N=\dcl(E,a)$, $N':=\pi^{-1}[N]$ and $L':=\acl(F)^{N'}$.
Our goal is to show that $L'=\dcl(F,a)$.

Of course $\CG(F)\cdot a\subseteq\CG(E)\cdot a$. Actually, we even have $\CG(F)\cdot a=\CG(E)\cdot a$: to see this we note that $\ulcorner\CG(F)\cdot a\urcorner\in F\cap \acl(E)=E$ (the last equality holds, since $E\subseteq F$ is regular), which means that $\CG(F)\cdot a$ is $E$-invariant. Thus $\CG(F)\cdot a$ must be equal to the orbit $\CG(E)\cdot a$.

Therefore $L_{a}:=\dcl(F,\CG(F)\cdot a)=\dcl(F,a)$. We need to show that $L_{a}=L'$. Since $F\subseteq\acl(F)^{N'}$ and $a\in\acl(F)^{N'}$, we obtain that $L_{a}\subseteq L'$. Consider the following diagram 
$$\xymatrix{
N' \ar[d]_{\pi} \ar[r]^{\subseteq} & \CG(F) \ar[d]_{\pi} \ar[r] & \CG(F)/N' \ar[d]^{\cong} \ar[r]^{\cong} & \CG(L'/F) \ar[d]^{\res} \ar[r]^{\res} & \CG(L_{a}/F) \ar[dl]^{\res}\\
N \ar[r]_{\subseteq} & \CG(E) \ar[r] & \CG(E)/N \ar[r]_{\cong} & \CG(L/E) &
}$$
We see that $\CG(L'/L_{a})=\ker\big(\res:\CG(L'/F)\to\CG(L_{a}/F)\big)=\{1\}$, hence by the Galois correspondence we get that $L'=L_{a}$.

Now, we will argue that $\pi$ is sorted. We start with $N\in\CN(\CG(E))$ and $J\in \mathcal{S}^{<\omega}$ such that $\res:\CG(L/E)\to\aut_J(L/E)$, where $L:=\acl(E)^N$, is an isomorphism.
Similarly as in the proof of Fact \ref{fact:implications_sorted}, we can find elements $a_1,\ldots,a_l\in S_J(\FC)$, for $l=[\CG(E):N]$, such that $L=\dcl(E,a_1,\ldots,a_l)$. Hence $c:=\ulcorner(a_1,\ldots,a_l)\urcorner\in e(L/E)\subseteq e(L'/F)$, where $L':=\acl(F)^{\pi^{-1}[N]}$. Therefore 
$$L'=\dcl(F,c)=\dcl(F,a_1,\ldots,a_l)$$
and so $\res:\CG(L'/F)\to\aut_J(L'/F)$ is an isomorphism, so $J\in\CF(\pi^{-1}[N])$.
\end{proof}

Now, notice that if in Lemma \ref{lemma2031}, instead of assuming that $\varphi_i$'s are sorted we assume they are $U$-sorted for some fixed $U$, then the proof of the lemma goes through with the only changes being that $b^*$ obtained in Claim 1b) belongs to $U(S,n)$ for a suitable $n$ rather than to $S$, and $T$ is $U$-sorted rather than sorted. 
It follows that the main results of Section \ref{sec:non_saturated_PAC}  apply to weakly sorted isomorphisms:
\begin{prop}\label{weakly_prop}
Lemma \ref{lemma2031} remains true if we replace the assumption of sortedness by $U$-sortedness for arbitrary fixed $U$.
Theorem \ref{thm:elementary_invariance}, Corollary \ref{final_cor}, and Remark \ref{remark:weakly.sorted.big.lemma} all remain true if we replace the assumption of sortedness by weak sortedness. In particular, in case of a finitely sorted language, the assumption of sortedness can be removed there.
\end{prop}
\end{appendix}
\bibliographystyle{plain}
\bibliography{1nacfa2}
\end{document}